\documentclass[reqno]{amsart}
\usepackage{amsmath, amsthm, amssymb, xypic, enumitem}
\usepackage[nobysame, non-sorted-cites, alphabetic]{amsrefs}
\usepackage{stmaryrd} 
\usepackage{hyperref}
\usepackage{tikz}
	\usetikzlibrary{cd}
\usepackage{color}

\setenumerate{label=\textnormal{(\arabic*)}} 

\newcommand{\cat}[1]{\ensuremath{\mathsf{#1}}}

\newcommand{\C}{\mathcal{C}}
\newcommand{\catC}{\mathcal{C}}
\newcommand{\op}{\ensuremath{^\mathrm{op}}}

\newcommand{\str}{\ensuremath{_\mathrm{str}}}

\DeclareMathOperator{\Ab}{\cat{Ab}}

\DeclareMathOperator{\Sub}{Sub}

\DeclareMathOperator{\Set}{\cat{Set}}

\DeclareMathOperator{\Par}{\cat{Par}}
\DeclareMathOperator{\dom}{dom}

\DeclareMathOperator{\Mon}{\cat{Mon}}
\DeclareMathOperator{\Grp}{\cat{Grp}}

\newcommand{\tightbox}[1]{
  \tikz{\node[shape=rectangle,draw,inner sep=0.25pt,line width=0.5pt] (A) {$#1$};}
}
\newcommand{\boxwedge}{\mathbin{\tightbox{\wedge}}}

\DeclareMathOperator{\HMag}{\cat{HMag}}
\DeclareMathOperator{\uHMag}{\cat{uHMag}}

\DeclareMathOperator{\cHMag}{\cat{cHMag}}
\DeclareMathOperator{\cuHMag}{\cat{cuHMag}}

\DeclareMathOperator{\Msc}{\cat{Msc}}
\DeclareMathOperator{\cMsc}{\cat{cMsc}}
\DeclareMathOperator{\Can}{\cat{Can}}
\DeclareMathOperator{\HGrp}{\cat{HGrp}}
\DeclareMathOperator{\HMon}{\cat{HMon}}

\DeclareMathOperator{\wHMag}{\cat{wHMag}}

\DeclareMathOperator{\MRing}{\cat{MRing}}
\DeclareMathOperator{\HRing}{\cat{HRing}}

\DeclareMathOperator{\pGrp}{\cat{pGrp}}
\DeclareMathOperator{\Crowd}{\cat{Crowd}}

\DeclareMathOperator{\Bim}{Bim}

\DeclareMathOperator{\Clos}{\cat{Clos}}
\DeclareMathOperator{\Mat}{\cat{Mat}}
\DeclareMathOperator{\sMat}{\cat{sMat}}
\DeclareMathOperator{\pMat}{\cat{pMat}}
\DeclareMathOperator{\Proj}{\cat{Proj}}

\DeclareMathOperator{\si}{si}

\DeclareMathOperator{\lMod}{\!-\cat{Mod}}

\DeclareMathOperator{\Hom}{Hom}

\DeclareMathOperator{\id}{id}

\DeclareMathOperator{\colim}{colim}

\newcommand{\Z}{\mathbb{Z}}

\newcommand{\F}{\mathbb{F}}
\newcommand{\FF}{\mathbf{F}}

\newcommand{\K}{\mathbf{K}}

\renewcommand{\P}{\mathcal{P}}

\newcommand{\E}{\mathcal{E}}

\newcommand{\two}{\mathbf{2}}

\newcommand{\ostar}{\circledast}

\newcommand{\separate}{\bigskip}

\numberwithin{equation}{section}

\theoremstyle{plain}
\newtheorem{theorem}[equation]{Theorem}
\newtheorem{corollary}[equation]{Corollary}
\newtheorem{lemma}[equation]{Lemma}
\newtheorem{proposition}[equation]{Proposition}

\theoremstyle{definition}
\newtheorem{definition}[equation]{Definition}

\newtheorem{example}[equation]{Example}
 
\newtheorem{remark}[equation]{Remark}

\begin{document}

\title[Categories of hypermagmas and hypergroups]{Categories of hypermagmas, hypergroups, and related hyperstructures}
\author{So Nakamura}
\email{snakamu2@uci.edu}
\author{Manuel L. Reyes}
\email{mreyes57@uci.edu}
\address{Department of Mathematics\\ University of California, Irvine \\
340 Rowland Hall\\ Irvine, CA 96267--3875 \\ USA}

\date{April 15, 2025}
\thanks{This work was supported by NSF grant DMS-2201273}

\begin{abstract}
In order to diagnose the cause of some defects in the category of canonical hypergroups, we investigate several categories of hyperstructures that generalize hypergroups.
By allowing hyperoperations with possibly empty products, one obtains categories with desirable features such as completeness and cocompleteness, free functors, regularity, and closed monoidal structures. We show by counterexamples that such constructions cannot be carried out within the category of canonical hypergroups. This suggests that (commutative) unital, reversible hypermagmas---which we call \emph{mosaics}---form a worthwhile generalization of (canonical) hypergroups from the categorical perspective. Notably, mosaics contain pointed simple matroids as a subcategory, and projective geometries as a full subcategory. 
\end{abstract}

\maketitle

\setcounter{tocdepth}{2} 
\makeatletter
\def\l@subsection{\@tocline{2}{0pt}{2.5pc}{5pc}{}}
\def\l@subsubsection{\@tocline{2}{0pt}{5pc}{7.5pc}{}} 
\makeatother
\tableofcontents

\section{Introduction}

Hypergroups~\cite{Marty} are a generalization of groups that allow for the product of two elements to be a (nonempty) set of elements, while hyperrings and hyperfields~\cite[\S 3]{Krasner} are ring-like objects whose additive structure is a particular kind of hypergroup. 
Although these structures were defined several decades ago, they have recently seen a flurry of renewed activity as they have been integrated into several different mathematical topics, including:
\begin{itemize}
\item tropical geometry~\cite{Viro:hyperfields, Jun:hyperfields, Lorscheid:hyperfield},
\item number theory and the field with one element~\cite{ConnesConsani:arithmetic, ConnesConsani:adele},
\item algebraic geometry~\cite{Jun:group, Jun:geometry}, and
\item matroid theory~\cite{BakerBowler:matroids, BakerLorscheid:moduli}.
\end{itemize}
(Several classical examples of hypergroups arising from group theory are also recalled in Subsection~\ref{sub:examples} below.)
These developments suggest that many future applications of the methods of hyperstructures are waiting to be revealed. 

In the spirit of aiding such cross-disciplinary relationships and spurring new ones, the goal of this paper is to analyze hypergroups and related structures within a categorical context, which we hope will allow for a clearer study and application of these objects across mathematical contexts.
A handful of papers have been devoted to various categories of hypermodules, such as~\cite{Madanshekaf:exact, Mousavi:free, JunSzczensyTolliver:protoexact, TenorioRoberto:category}, whose underlying additive structure forms a canonical hypergroup. 
In the case of (ordinary) modules over a ring $R$, many desirable properties of the category $R \lMod$ follow from those of the category $\Ab$ of abelian groups, perhaps most notably the property of being an abelian category. Thus we began this project with the goal of understanding the category of canonical hypergroups.

In our own investigations of the categories of (canonical) hypergroups, we were surprised to find that certain desirable properties do not hold for these categories, but \emph{do} hold for categories of more general hyperstructures. At the heart of this misbehavior of hypergroups seems to be the assumption that the product of any two elements returns a \emph{nonempty} set. While it is not initially obvious, this is in fact a consequence (assuming the other axioms) of the requirement that the product be associative (Lemma~\ref{lem:nondegenerate}). 

\begin{figure}
\begin{tikzcd}[row sep=small, column sep=small]
\cat{cMsc} \arrow[r, phantom, "\subseteq"] 
  & \cat{Msc} \arrow[r, phantom, "\subseteq"] 
  & \uHMag \arrow[r, hookrightarrow] 
  & \HMag \\  
\Can \arrow[r, phantom, "\subseteq"] \arrow[u, phantom, sloped, "\subseteq"] 
  & \HGrp \arrow[r, phantom, "\subseteq"] \arrow[u, phantom, sloped, "\subseteq"] 
  & \HMon \arrow[u, phantom, sloped, "\subseteq"]  
  & \\ 
\Ab \arrow[r, phantom, "\subseteq"] \arrow[u, phantom, sloped, "\subseteq"] 
  & \Grp \arrow[r, phantom, "\subseteq"] \arrow[u, phantom, sloped, "\subseteq"] 
  & \Mon \arrow[u, phantom, sloped, "\subseteq"]  
  & 
\end{tikzcd}
\caption{Categories of hyperstructures. A ``$\subseteq$'' denotes a full subcategory, while the ``$\hookrightarrow$'' is a faithful forgetful functor.}
\label{fig:inclusions} 
\end{figure}

For this reason we study categories of sets equipped with hyperoperations---called \emph{hypermagmas}---without any restriction on the subset returned by a hyperoperation. This is in line with the treatment of hypermagmas in~\cite{Dudzik:quantales}. We call the objects that generalize hypergroups in this setting by the name of \emph{mosaics}.\footnote{The term \emph{mosaic} is a play on the word \emph{group} as a collection of objects, but with an artistic twist. Instead of adding to the large number of structures whose names attach a prefix to the word \emph{group} (e.g., hypergroup, semigroup, quasigroup, multigroup, polygroup), we decided to choose a brand new term.} They are assumed to have an identity, unique inverses, and reversibility, but no associativity and with possibly empty products. Then commutative mosaics can be viewed as a nonassociative generalization of canonical hypergroups. We denote these categories as follows:
\begin{itemize}
\item $\HMag$, $\uHMag$, and $\HMon$ are the categories of (unital) hypermagmas and hypermonoids; 
\item $\Msc$ and $\cMsc$ are the categories of mosaics and commutative mosaics;
\item $\HGrp$, $\Can$, and $\HMon$ are the categories of hypergroups, canonical hypergroups, and hypermonoids;
\end{itemize}
while $\Mon$, $\Grp$, and $\Ab$ denote the ususal categories of monoids and (abelian) groups. See Figure~\ref{fig:inclusions} for the relationships between these various categories.

\separate

We now survey some of the results proved below.
Among the first categorical properties one would naturally ask about are certainly completeness and cocompleteness. The categories of (unital) hypermagmas and mosaics behave well in this respect, while the categoroes of hypergroups and canonical hypergroups fail to be either complete or cocomplete.

\begin{theorem}
The categories $\HMag$, $\uHMag$, $\Msc$, and $\cMsc$ are complete and cocomplete, and free objects exist in these categories. 

The categories $\HGrp$ and $\Can$ have small products, kernels, and cokernels. However, there are binary coproducts, equalizers, and coequalizers that do not exist in $\HGrp$ or $\Can$.
\end{theorem}

\begin{proof}
This combines Propositions~\ref{prop:free and cofree} and~\ref{prop:forgetful creates} and Theorem~\ref{thm:unital cocomplete} of subsection~\ref{sub:limits and colimits}; Theorems ~\ref{thm:reversible forgetful} and~\ref{thm:free reversible} of subsection~\ref{sub:convenient}; and Theorem~\ref{thm:hypergroup operations} along with Propositions~\ref{prop:bad coproduct}, \ref{prop:bad equalizer}, and~\ref{prop:bad coequalizer} of subsection~\ref{sub:hypergroups}.
\end{proof}

A more subtle aspect of these categories is the nature of various epimorphisms and monomorphisms. In contrast to the categories of (abelian) groups, there are various types of epimorphisms and monomorphisms that do not coincide in these categories. Recall that in any category, a \emph{regular} epimorphism (resp., monomorphism) is defined to be a coequalizer (resp., equalizer) of a pair of morphisms. Furthermore, in any category with a zero object, a \emph{normal} epimorphism (resp. monomorphism) is defined to be a cokernel (resp., kernel) of a morphism. 
We characterize these morphisms as follows.

\begin{theorem}
In each of the categories $\HMag$, $\uHMag$, $\Msc$, and $\cMsc$:
\begin{itemize}
\item The epimorphisms (resp., monomorphisms) are the surjective (resp., injective) morphisms;
\item The regular epimorphisms (resp., monomorphisms) can be characterized as the short (resp., coshort) morphisms (Definition~\ref{def:short}).
\end{itemize}
Each of these four categories is regular.
Furthermore, in the categories $\uHMag$, $\Msc$, and $\cMsc$, the normal monomorphisms correspond to strict absorptive subhypermagmas (Definitions~\ref{def:strict sub} and~\ref{def:absorptive}), and normal epimorphisms correspond to unitizations (Definition~\ref{def:unitization}). 
\end{theorem}

\begin{proof}
See subsection~\ref{sub:epi and mono} and Corollary~\ref{cor:mosaic morphisms}.
\end{proof}

The relationships between these morphisms is visualized in Figures~\ref{fig:monomorphisms} and~\ref{fig:epimorphisms}.

\begin{figure}[b]
\scriptsize
\begin{tikzcd}
\begin{tabular}{c} normal \\ monomorphism \end{tabular} \ar[r, Rightarrow] \ar[d, Leftrightarrow] & 
	\begin{tabular}{c} regular \\ monomorphism \end{tabular} \ar[r, Rightarrow] \ar[d, Leftrightarrow] & 
	\mbox{monomorphism} \ar[d, Leftrightarrow] \\
\begin{tabular}{c} strict absorptive \\ subhypermagma \end{tabular}  \ar[r, Rightarrow] &
	\begin{tabular}{c} coshort \\ morphism \end{tabular} \ar[r, Rightarrow] & 
	\begin{tabular}{c} injective \\ morphism \end{tabular}
\end{tikzcd}
\caption{Characterizations of various monomorphisms.}
\label{fig:monomorphisms}
\end{figure}

\begin{figure}[h]
\scriptsize
\begin{tikzcd}
\begin{tabular}{c} normal \\ epimorphism \end{tabular} \ar[r, Rightarrow] \ar[d, Leftrightarrow] & 
	\begin{tabular}{c} regular \\ epimorphism \end{tabular} \ar[r, Rightarrow] \ar[d, Leftrightarrow] & 
	\mbox{epimorphism} \ar[d, Leftrightarrow] \\
\mbox{unitization}  \ar[r, Rightarrow] & 
	\begin{tabular}{c} short \\ morphism \end{tabular} \ar[r, Rightarrow] & 
	\begin{tabular}{c} surjective \\ morphism \end{tabular}
\end{tikzcd}
\caption{Characterizations of various epimorphisms.}
\label{fig:epimorphisms}
\end{figure}

\separate

One other crucial aspect of the theory of abelian groups is the formation of tensor products and hom-groups, along with the tensor-hom adjunction.
In category-theoretic terms, the structure $(\Ab, \otimes_\Z, \Z)$ forms a closed monoidal category.
In this paper we are able to define a closed monoidal structure $(\cMsc, \boxtimes, \mathbf{F})$ on the category of commutative mosaics, which can be viewed as a replacement for the tensor product of abelian groups. Somewhat surprisingly, it turns out that the categories of hypermagmas and unital hypermagmas also have closed monoidal structures! These monoidal products are  constructed by a sequence of successive quotient objects
\[
M \boxdot N \twoheadrightarrow M \boxwedge N \twoheadrightarrow M \boxtimes N.
\]
The internal hom for each of these categories is defined in a natural way by endowing the ordinary hom set $\Hom(M,N)$ with a hyperoperation of the form
\begin{equation}\label{eq:natural}
f \star g = \{h \in \Hom(M,N) \mid h(x) \in f(x) \star g(x) \mbox{ for all } x \in M\}.
\end{equation}
In addition, we provide explicit counterexamples to show that one cannot hope to provide a similar ``tensor product'' for canonical hypergroups.

\begin{theorem}
Each of the categories $\HMag$, $\uHMag$, and $\cMsc$ has a closed monoidal structure such that:
\begin{itemize}
\item the monoidal product represents the functor of bimorphisms~\cite{Jagiello:tensor} on the category;
\item the monoidal unit is the free object generated by one element;
\item the internal hom is given by endowing the ordinary hom set with the natural hyperoperation~\eqref{eq:natural}.
\end{itemize}

However,
there exists a canonical hypergroup $H$ such that $\Can(\Z/2\Z, H)$ does not form a hypergroup under the natural hyperoperation~\eqref{eq:natural}. Furthermore, for the Klein four-group $V$, the functor of bimorphisms $\Bim_{\Can}(V,V;-)$ is not representable on $\Can$.
\end{theorem}

\begin{proof}
See subsection~\ref{sub:closed monoidal} along with Theorems~\ref{thm:mosaic monoidal}, \ref{thm:empty sum}, and~\ref{thm:Klein four}.
\end{proof}

The classical tensor-hom adjunction becomes particularly useful when viewing abelian groups as the underlying additive structure of rings. For instance, the structure of a ring $R$ can alternatively be encoded in terms of a monoid object in $\Ab$. We show that the closed monoidal structure on $\cMsc$ allows us to do something similar for multirings (including hyperrings). 

\begin{theorem}[Theorem~\ref{thm:multirings}]
The categories of multirings and hyperrings have fully faithful embeddings into the category of monoid objects of $\cMsc$.
\end{theorem}

The above suggests that monoid objects in $\cMsc$ are an interesting and potentially useful generalization of multirings and hyperrings, whose theory we hope to develop in the future.

Finally, we wish to show that there is a rich supply of commutative mosaics aside from canonical hypergroups. This is accomplished via a functor from pointed simple matroids to commutative mosaics. The construction is inspired by a similar hyperoperation defined on irreducible projective geometries, as in~\cite[Section~3]{ConnesConsani:adele}. In fact, this functor allows us to extend this operation to projective geometries~\cite{FaureFrolicher} that are not necessarily irreducible.

\begin{theorem}[Theorems~\ref{thm:matroid mosaics} and~\ref{thm:projective embedding}]
There is a faithful functor $\sMat_\bullet \to \cMsc$ from the category of simple pointed matroids to the category of commutative mosaics. This functor induces a fully faithful embedding from the category of projective geometries to the category of commutative mosaics.
\end{theorem}

In light of the above results, it seems reasonable to consider commutative mosaics as a ``convenient category of canonical hypergroups,'' borrowing well known terminology from~\cite{Steenrod}.
Our hope is that this framework will provide a flexible context in which to study representations of hyperrings, multirings, and related structures. 
In future work, we will examine the properties of generalized hyperrings whose underlying additive structure forms a commutative mosaic, as well as their appropriate categories of modules. This will allow us to examine the category of representations of an ordinary ring over a base hyperfield. 
We are particularly interested in representations of rings over the Krasner hyperfield, which we expect could provide a novel ``base-free'' representation theory for general rings.

The results above suggest an intriguing question. The category of mosaics gains its advantages by omission of the associative axiom in the objects. Is there a useful subcategory of ``weakly associative'' objects in $\Msc$ (or $\cMsc$) that retains good categorical properties while keeping some measure of algebraic constraint on the structures?

\separate

\textbf{Acknowledgments.} We wish to thank Matt Satriano for a number of stimulating discussions that led to the present work, Oliver Lorscheid for informing us about crowds and the reference~\cite{LorscheidThas}, and Siddhant Jajodia and the anonymous referee for several corrections and suggestions. We also thank Zo\"{e} Reyes for suggesting the term \emph{mosaic}.

\section{A brief overview of hyperstructures}

Let $M$ be a set, and let $\P(M)$ denote its power set. A \emph{hyperoperation} on $M$ is a function
\[
\star \colon M \times M \to \P(M).
\]
A hyperoperation extends to a binary operation on the power set
\[
\star \colon \P(M) \times \P(M) \to \P(M)
\]
in the obvious way: given $X,Y \subseteq M$ we set
\[
X \star Y = \bigcup_{(x,y) \in X \times Y} x \star y.
\]
Notice that by this definition, for all $X \subseteq M$ we have
\begin{equation}\label{eq:times empty}
X \star \varnothing = \varnothing = \varnothing \star X.
\end{equation}
In case $x \star y = \{z\}$ is a singleton, it is customary to instead write the shorthand equation
$x \star y = z$. 
We also extend this shorthand to operations on subsets $Y \subseteq M$, by setting $x \star Y = \{x\} \star Y$. 
In general, we will often use an element interchangeably with the singleton containing that element throughout this paper.

The following terminology is inspired by~\cite[D\'{e}finition~1.2]{Mittas} and~\cite[Definition~2.17]{Dudzik:quantales}.

\begin{definition}\label{def:hypermagma}
A \emph{hypermagma} $(M,\star)$ is a set endowed with a hyperoperation. A function $f \colon M \to N$ between hypermagmas is said to be:
\begin{itemize}
\item a \emph{(colax) morphism} if it satisfies $f(x \star y) \subseteq f(x) \star f(y)$ for all $x,y \in M$;
\item a \emph{lax morphism} if it satisfies $f(x \star y) \supseteq f(x) \star f(y)$ for all $x,y \in M$;
\item a \emph{strict morphism} if it is both a lax and colax morphism, satisfying $f(x \star y) = f(x) \star f(y)$ for all $x,y \in M$.
\end{itemize}
We let $\HMag$ denote the category of hypermagmas with colax morphisms, which we simply refer to as \emph{morphisms} of hypermagmas.
\end{definition}

On occasions where we wish to emphasize the hypermagma $M$ to which a hyperoperation $\star$ belongs, we will use the notation $\star_M = \star$.

Unlike much of the literature on hypergroups, we wish to allow for the possibility that $x \star y = \varnothing$ in certain cases. For this reason, it is convenient to introduce a relation $\odot \subseteq M \times M$ by setting
\[
x \odot y \iff x \star y \neq \varnothing.
\]
We will say that $\star$ is a \emph{total} hyperopration if $M \neq \varnothing$ and $\odot = M \times M$ (i.e., $x \star y \neq \varnothing$ for all $x,y \in M$).
Note that if $\star$ is total, and if we let $\P(M)^*$ denote the set of nonempty subsets of $G$, then $\star$ corestricts to an operation 
\[
\star \colon M \times M \to \P(M)^*,
\]
which induces a binary operation $\star \colon \P(M)^* \times \P(M)^* \to \P(M)^*$.

\separate

Let $M$ be a hypermagma.  We say that an element $e \in M$ is a  \emph{weak identity} if, for all $x \in M$, 
\[
x \in e \star x \, \cap \, x \star e.
\]
Note that a weak identity need not be unique. For example, endow any nonempty set $X$ with the hyperproduct is given by $x \star y = X$ for all $x,y \in X$. (This will be shown in Proposition~\ref{prop:free and cofree} to be a cofree hypermagma.) Then every element of $X$ is a weak identity, so that uniqueness fails when $X$ has at least two elements. 

If the stronger condition
\[
x = e \star x = x \star e
\]
holds for all $x \in M$, then we will say that $e$ is an \emph{identity} for $G$. In this case, a familiar argument shows that an identity is unique if it exists. (In the literature on hypergroups, what we call a weak identity is typically called an \emph{identity}, while what we call an identity is referred to as a \emph{scalar} identity. See~\cite{CorsiniLeoreanu:hyperstructure} for more details.) 
We refer to a hypermagma with identity as a \emph{unital} hypermagma, and we let $\uHMag$ denote the category of unital hypermagmas with unit-preserving morphisms.

\separate

Let $M$ be a hypermagma with identity $e$, and fix $x \in M$. An element $x' \in M$ is an \emph{inverse} for $x$ if $e \in x \star x' \, \cap \, x' \star x$. Note that inverses of this sort can be far from unique. (For instance, given any set $X$ endow $M = X \sqcup \{e\}$ with a hyperoperation such that $e$ is an identity and $x \star y = M$ for all $x,y \in X$; then any two nonidentity elements are inverse to one another.) Thus we will reserve the notation $x^{-1}$ for the situations in which $x$ has a unique inverse. One can easily verify that the identity is its own unique inverse, so that we are justified in writing $e^{-1} = e$.

\separate

Beyond simply demanding the uniqueness of inverses, there is a more principled assumption that is typically employed in the study of hypergroups. Roughly speaking, the idea is to require that inverses allow us to ``solve for'' elements that appear in products (or sums).
A hypermagma $M$ is said to be \emph{reversible} if there is a function $(-)^{-1} \colon M \to M$ such that
\[
x \in y \star z \implies y \in x \star z^{-1} \mbox{ and } z \in y^{-1} \star x
\]
for all $x, y, z \in M$. (Although it was not endowed with a name, a version of this property was already assumed in the seminal work of Marty~\cite[p.~46]{Marty}.)
In the case where $M$ has an identity $e$, 
reversibility implies the existence of unique inverses: for $x \in x \star e$ implies $e \in x \star x^{-1}$ and $e \in x^{-1} \star x$, while if $x'$ is any other inverse for $x$ then $e \in x x'$ implies that $x' \in x^{-1} \star e$ so that $x' = x^{-1}$. It follows that $(-)^{-1}$ is an involution, and one can check that the reversibility implication can be strengthened to the following equivalence for all $x,y,z \in M$:
\begin{align*}
x \in y \star z  &\iff y \in x \star z^{-1} \\
 &\iff z \in y^{-1} \star x.
\end{align*}

If $M$ and $N$ are reversible hypermagmas with identity and $f \in \uHMag(M,N)$, then by uniqueness of inverses we may deduce that $f(m^{-1}) = f(m)^{-1}$ for all $m \in M$. Thus the reversible structure is automatically preserved by unital morphisms. 
Reversible hypergroups with identity form a particularly important category for our considerations. For convenience, we introduce the following terminology.

\begin{definition}
A \emph{mosaic} $(M, \star, e)$ is a hypermagma with identity that is also reversible. We let $\Msc$ denote the full subcategory of $\uHMag$ whose objects are mosaics, and we let $\cMsc$ denote the full subcategory of commutative mosaics. 
\end{definition}

Our commutative mosaics (including canonical hypergroups) will typically be written additively $(M,+,0)$, with the additive inverse of $y \in M$ written as $-y$ and $x + (-y) = x-y$.
In the commutative case, the reversibility axiom takes the simplified form 
\[
x \in y+z \implies z \in x - y.
\]

Given any hypermagma $M$, one can define its \emph{opposite} hypermagma in the familiar way: as a set $M\op = \{m\op \mid m \in M\}$ with hypermultiplication
\[
x\op \star y\op = (y \star x)\op := \{z\op \mid z \in y \star x\}.
\]
As in the case of groups, the inversion of a mosaic gives an isomorphism of $M$ with its opposite, or an anti-isomorphism, thanks to the next lemma. For a subset $S \subseteq M$ of a mosaic, we will use the notation 
\[
S^{-1} = \{s^{-1} \mid s \in S\}.
\]

\begin{lemma}\label{lem:inversion}
If $M$ is a mosaic and $x,y \in M$, then
\[
(x \star y)^{-1} = y^{-1} \star x^{-1}.
\]
\end{lemma}

\begin{proof}
Let $z \in M$. Then $z \in (x \star y)^{-1}$ if and only if $z^{-1} \in x \star y$. By reversibility,
\[ 
z^{-1} \in x \star y \iff y \in x^{-1} \star z^{-1} \iff x^{-1} \in y \star z \iff z \in y^{-1} \star x^{-1}. \qedhere
\]
\end{proof}

\separate

Finally we arrive at hypergroups, the original motivation for this study. We will work with the following definition of hypergroups. Although it is a bit stronger than some definitions given in the literature, it is quite close to Marty's original definition~\cite{Marty}, with the only difference being the requirement of an identity. (This is what Marty called a ``completely regular'' hypergroup). This is also the convention used in~\cite{Zieschang}. 

A hypermagma $(M,\star)$ is said to be \emph{associative} if it satisfies 
\[
x \star (y \star z) = (x \star y) \star z
\]
for all $x,y,z \in M$.

\begin{definition}
A \emph{hypermonoid} is an associative hypermagma with identity.
A \emph{hypergroup}  $(G, \star, e)$ is a hypermonoid that is total and reversible.
A \emph{canonical hypergroup} $(G, +, 0)$ is a commutative reversible hypergroup. 
We let $\HMon$ denote the full subcategory of $\uHMag$ whose objects are the hypermonoids, while $\HGrp$ and $\Can$ respectively denote the full subcategories of hypergroups and canonical hypergroups.
\end{definition}

The standard definition for hypergroups requires that the hyperoperation be total. In practice, we have found it easy to accidentally overlook this condition. The following fact alleviates this problem. A version of this was remarked in~\cite[p.~168]{Mittas}. 

\begin{lemma}\label{lem:nondegenerate}
Let $(G,\star)$ be a hypermagma. Suppose that $G$ is associative and that there exists $z \in G$ such that:
\begin{itemize}
\item $x \star z \neq \varnothing$ for all $x \in G$;
\item for all $y \in G$, there exists $y' \in G$ such that $z \in y \star y'$.
\end{itemize}
Then $\star$ is total. 
In particular, a mosaic is a hypergroup if and only if it is associative.
\end{lemma}

\begin{proof}
Note that $G \neq \varnothing$ because $z \in G$.
Let $x,y \in G$, and fix $y' \in G$ such that $z \in y \star y'$. Note that
\[
\varnothing \neq x \star z \subseteq x \star (y \star y') = (x \star y) \star y'.
\]
So $(x \star y) \star y'$ is nonempty. It follows from~\eqref{eq:times empty}
that $x \star y \neq \varnothing$, proving that $\star$ is total.

Now suppose that $(G, \star, e)$ is an associative mosaic, with $\star$ not necessarily total. Then $z = e$ and the elements $y' = y^{-1}$ for each $y \in G$ satisfy the hypotheses above. It follows that $\star$ is total and $G$ is a hypergroup. 
\end{proof}

Another situation in which this lemma may be useful is when $G$ is a hypermagma with a \emph{zero} (or \emph{absorbing}) element $0 \in G$, in the sense 
that $x \cdot 0 = 0 = 0 \cdot x$ for all $x \in G$. For if we set $z = 0$ and all $y' = 0$, then the
hypotheses of Lemma~\ref{lem:nondegenerate} are satisfied.

\separate

We remark that if $\C$ is any of the categories of hyperstructures pictured in Figure~\ref{fig:inclusions}, then we let $\C\str$ denote the wide subcategory of $\C$ whose morphisms are the \emph{strict} morphisms as in Definition~\ref{def:hypermagma}.

\subsection{Substructures of hyperstructures and associated morphisms}

The theory of subobjects in these categories is subtle and will be revisited in Subsection~\ref{sub:epi and mono} below. For the moment, we introduce the following definitions in an effort to distinguish between different types of injective morphisms.

\begin{definition}\label{def:strict sub}
Let $M$ be a hypermagma. A \emph{strict subhypermagma} of $M$ is a subset $L \subseteq M$ such that, for any $x,y \in L$, we have $x \star y \subseteq L$. If $M$ is a unital hypermagma (resp., mosaic, hypergroup), we define a \emph{strict unital subhypermagma} (resp., \emph{strict submosaic} or \emph{strict subhypergroup}) to be a strict subhypermagma that also contains the identity (resp., and closed under taking inverses).
\end{definition}

The reasoning behind this terminology is that strict subhypermagmas (mosaics, etc.) $L$ of $M$ exactly correspond to injective strict morphisms $i \colon L \to M$ of hypermagmas (mosaics, etc.).

On the other hand, every subset of a hypermagma induces an injective morphism in the following way.

\begin{definition}
Let $M$ be a hypermagma and let $L \subseteq M$ be a subset. Define a hyperoperation on $L$ by
\[
x \star_L y = (x \star_M y) \cap L.
\]
We refer to $(L, \star_L)$ as a \emph{weak subhypermagma} of $L$. Given such a weak subhypermagma, we furthermore define the following:
\begin{itemize}
\item If $M$ is unital, then $L$ is a \emph{weak unital subhypermagma} if $e_M \in L$. 
\item If $M$ is a mosaic, then $L$ is a \emph{weak submosaic} if $e_M \in L$ and $L$ is closed under formation of inverses.
\end{itemize}
\end{definition}

It is clear that if $L$ is a weak subhypermagma (resp., unital subhypermagma, submosaic) of $M$, then the inclusion function $i \colon L \to M$ is a morphism in $\HMag$ (resp., $\uHMag$, $\Msc$).
Note that if $M$ is an associative hypermagma, then any strict subhypermagma of $M$ is automatically associative. However, there is no reason for a weak subhypermagma of $M$ to remain associative.

\separate

We take this opportunity to define two properties of morphisms that will be of importance throughout this paper, one of which corresponds to weak subhypermagmas.

\begin{definition}\label{def:short}
A morphism of hypermagmas $p \colon M \to N$ is \emph{short} if it is surjective and satisfies
\begin{equation}\label{eq:short}
x \star y = p(p^{-1}(x) \star p^{-1}(y))
\end{equation}
for all $x,y \in N$.  
Dually, a morphism $i \colon L \to M$ is \emph{coshort} if it is injective and satisfies
\begin{equation}\label{eq:coshort}
i^{-1}(i(x) \star i(y)) = x \star y
\end{equation}
for all $x,y \in L$.
\end{definition}

It is straightforward to check that a coshort morphism $i \colon L \to M$ is the same as an isomorphism of $L$ onto the weak subhypermagma $i(L)$ of $M$. This makes it easy to verify that coshort morphisms are closed under composition. Notice that the containment ``$\supseteq$'' of~\eqref{eq:coshort} holds for any morphism $i$, so that the condition is equivalent to 
\[
i^{-1}(i(x) \star i(y)) \subseteq x \star y.
\]

Dually, short morphisms are closed under composition; indeed, if $L \overset{p}{\twoheadrightarrow} M \overset{q}{\twoheadrightarrow} N$ is a sequence of short morphisms, then for $x,y \in N$ we have
\begin{align*}
x \star y &= q(q^{-1}(x) \star q^{-1}(y)) \\
&= q(p(p^{-1}(q^{-1}(x))\star p^{-1}(q^{-1}(y)))) \\
&= (q \circ p) ((q \circ p)^{-1}(x) \star (q \circ p)^{-1}(y)).
\end{align*}
Once again the containment ``$\supseteq$'' of~\eqref{eq:short} always holds, so that the condition is equivalent to
\[
x \star y \subseteq p(p^{-1}(x) \star p^{-1}(y)).
\]

The following shows that the conditions of injectivity and surjectivity in the definition above are redundant in the unital case.

\begin{lemma}
Let $L$, $M$, and $N$ be unital hypermagmas, and let $i \in \uHMag(L,M)$ and $p \in \uHMag(M,N)$. Then $p$ is short if and only if it satisfies~\eqref{eq:short}, and $i$ is coshort if and only if it satisfies~\eqref{eq:coshort}.
\end{lemma}

\begin{proof}
Suppose that $p$ is unital and satisfies~\eqref{eq:short}. To see that $p$ is surjective, let $x \in N$. Then
\[
x = x \star e_N = p(p^{-1}(x) \star p^{-1}(e_N))) 
\]
shows that $x$ is in the image of $p$. So $p$ is surjective and thus is short.

Now suppose that $i$ is unital and satisfies~\eqref{eq:coshort}. Fix $x \in L$. Then
\[
i^{-1}(i(x)) = i^{-1}(i(x) \star e_M) \subseteq x \star e_M = x
\]
shows that there is a unique element mapping to $i(x)$ under~$i$. So $i$ is injective and therefore short.

The converse implications are trivial.
\end{proof}

We will revisit short and coshort morphisms in Subsections~\ref{sub:epi and mono} and~\ref{sub:convenient}, which will reveal their role as regular epimorphisms and monomorphisms. 
In the meantime, we record the following result describing properties that pass to ``quotient'' hypermagmas via short or strict surjective morphisms.

\begin{lemma}\label{lem:short image}
Let $p \colon M \to N$ be a surjective morphism of hypermagmas. Then:
\begin{enumerate}
\item If $p$ is strict, then $p$ is short.
\item If $p$ is short and $M$ is commutative, then $N$ is commutative.
\item If $p$ is strict and $M$ is associative, then $N$ is associative.
\end{enumerate}
\end{lemma}

\begin{proof}
(1) Let $x,y \in N$. Because $p$ is surjective and strict, we may compute
\[
x \star y = p(p^{-1}(x)) \star p(p^{-1}(y)) = p(p^{-1}(x) \star p^{-1}(y)).
\]
So $p$ is short.

(2) If $M$ is commutative and $x,y \in N$, then shortness gives
\[
x \star y = p(p^{-1}(x) \star p^{-1}(y)) = p(p^{-1}(y) \star p^{-1}(x)) = y \star x.
\]
Thus $N$ is commutative.

(3) Assume $M$ is associative and $p$ is strict.
 Given $x,y,z \in N$, we claim that
\[
(x \star y) \star z = p(p^{-1}(x) \star p^{-1}(y) \star p^{-1}(z)) = x \star (y \star z).
\]
Using the fact that $p$ is strict and surjective, we have
\[
p(p^{-1}(x) \star p^{-1}(y) \star p^{-1}(z)) = p(p^{-1}(x) \star p^{-1}(y)) \star p(p^{-1}(z)) = (x \star y) \star z.
\]
This verifies the first equality above, and the second follows by a symmetric argument. Thus $M$ is associative.
\end{proof}

\subsection{Some examples of hypergroups and hypermagmas}\label{sub:examples}

Before passing to the level of categories of hyperstructures, we pause to mention a few sources of examples of hypermagmas and hypergroups. 

We begin with some classical examples arising from ordinary group theory. Two of these are obtained as certain quotients of a group in the following way. Let $G$ be a group, and let $\sim$ be an equivalence relation on $G$ such that the equivalence classes satisfy the following properties for all $a,b \in G$ and the identity $e \in G$:
\begin{enumerate}[label=(\roman*)]
\item The setwise product satisfies $[e][a] = [a]=[a][e]$,
\item Setwise inversion satisfies $[a]^{-1} = [a^{-1}]$.
\item The setwise product $[a][b]$ of any two equivalence classes is a union of equivalence classes.
\end{enumerate}
Then we define a hyperoperation on $G/{\sim}$ by
\begin{equation}\label{eq:group quotient}
[a] \star [b] = \{[c] \mid [c] \subseteq [a][b]\} = \{[c] \mid c \in [a][b]\}.
\end{equation}
This is easily verified from conditions (i)--(iii) to be reversible hypermagma with identity $[e]$ and $[a]^{-1} = [a^{-1}]$. It is also associative since
\[
[a] \star ([b] \star [c]) = \{[w] \mid w \in [a][b][c]\} = ([a] \star [b]) \star [c].
\]
Thus $G/{\sim}$ is a hypergroup, and the observation that $ab \in [a][b]$ for all $a,b \in G$ ensures that canonical surjection $G \twoheadrightarrow G/{\sim}$ sending $a \mapsto [a]$ is a morphism of hypergroups.

\begin{example}
Let $G$ be a group and let $K$ be a subgroup of $G$. The equivalence relation $\sim$  induced by the double coset partition of $G$ satisfies conditions~(i)--(iii). The resulting quotient $G/\!/K$ is the double coset space, and the hyperoperation~\eqref{eq:group quotient} in this case takes the form
\begin{align*}
KaK \star KbK &= \{KcK \mid KcK \subseteq (KaK)(KbK)\} \\
&= \{KcK \mid c \in KaKbK\}.
\end{align*}
This makes $G/\!/K$ into a hypergroup (see also~\cite[p.~720]{DresherOre}) with identity $KeK = K$ and inverses $(KaK)^{-1} = Ka^{-1}K$, which coincides with the ordinary quotient group in case $K$ is a normal subgroup.
\end{example}

\begin{example}
For any group $G$, the equivalence relation induced by the action of conjugation satisfies conditions (i)--(iii) above. The corresponding quotient $\overline{G}$ is the set of conjugacy classes in $G$; let $C_g$ denote the conjugacy class of $g \in G$. In this case the hyeroperation~\eqref{eq:group quotient} is
\[
C_a \star C_b = \{C_g \mid C_g \subseteq C_a \cdot C_b\} = \{C_g \mid g \in C_a C_b\}.
\]
Using $xC_a = xC_a x^{-1} x = C_a x$ for all $a,x \in G$, we have $C_a C_b = C_b C_a$ and therefore $C_a \star C_b = C_b \star C_a$ for all $a,b \in G$. It follows as in~\cite{Dietzman} that $\overline{G}$ is a canonical hypergroup.
\end{example}

\begin{example}\label{ex:group action}
Let $A$ be an abelian group and let $G$ be a group acting by automorphisms on $A$. The orbit equivalence relation on $A$ satisfies conditions~(i)--(iii), so that the quotient $A/G$ becomes a canonical hypergroup under the hyperoperation
\[
Ga + Gb = \{Gc \mid c \in Ga + Gb\}.
\]
This construction forms the underlying additive hypergroup of quotient hyperrings and hyperfields as introduced by Krasner in~\cite{Krasner:quotient}.
\end{example}

\begin{example}
Now suppose that $G$ is a finite group, and let $\widehat{G} = \{\chi_i\}$ denote the set of irreducible complex characters of $G$. Given $\phi, \psi \in \widehat{G}$, their product is a character and can be expressed as a nonnegative sum of irreducible characters. We obtain a hyperoperation by setting
\[
\phi \star \psi = \left\{\chi_i \mid n_i \neq 0 \mbox{ in } \phi \cdot \psi = \sum n_i \chi_i \right\}.
\]
As discussed in~\cite[Section~2]{Roth:character}, this gives $\widehat{G}$ the structure of a canonical hypergroup, whose identity is the character of the trivial representation and with inverses given by complex conjugates.
\end{example}

Many of the hypergroups and hypermagmas of interest for us do not originate from group theory. Below are a few other examples of hypergroups.

\begin{example}\label{ex:Krasner}
The \emph{Krasner hyperfield} $\K = \{0,1\}$ is the hyperring with the uniquely determined multiplication 
(where $0$ is the multiplicative zero element and $1$ is the multiplicative identity) and with addition
given by letting $0$ be the additive identity and setting
\[
1+1 = \{0,1\}.
\]
This is easily seen to be a hyperfield~\cite{ConnesConsani:adele}. In Section~\ref{sec:mosaics} we will be interested in the additive hypergroup structure of $\K$. 
\end{example}

\begin{example}\label{ex:modular}
Let $(L,\wedge)$ be a meet-semilattice. We may define a commutative hyperoperation on $L$ as follows: for $a,b \in L$, 
\[
a \owedge b = \{c \in L \mid a \wedge c = b \wedge c = a \wedge b \}.
\]
Under the inversion given by $a^{-1} = a$ for all $a \in L$, it is easily verified that this is a reversible hypermagma. Thus if $L$ has a top element $1$, then $(L, \owedge, 1)$ is a commutative mosaic.

Now assume that $(L,\vee,\wedge,0,1)$ is in fact a bounded lattice. Nakano showed~\cite[Theorem~1]{Nakano} that $(L, \owedge, 1)$ forms a hypergroup if and only if $L$ is modular (see also~\cite[Section~4.3]{CorsiniLeoreanu:hyperstructure}). (Recall that a lattice $(L, \vee, \wedge, 0, 1)$ is \emph{modular} if it satsfies the modular law: for all $a,b,c \in L$,
if $a \leq c$ then $a \vee (b \wedge c) = (a \vee b) \wedge c$.)
\end{example}

The following example illustrates that hyperoperations are flexible enough to include partial binary operations and their resulting partial structures.

\begin{example}
We recall from~\cite[Section~3.2]{HeunenReyes:active} that a \emph{partial group} is a set $G$ equipped with a reflexive, symmetric binary relation $\odot \subseteq G \times G$ of \emph{commeasurability}, with a globally defined unary operation $(-)^{-1} \colon G \to G$, a partially defined operation $\ast \colon \odot \to G$, and an element $e \in G$ such that any set $S \subseteq G$ of pairwise commeasurable elements is contained in a pariwise commeasurable set $A \subseteq G$ containing $e$ on which $\star$ restricts to give the structure of an abelian group. The category $\pGrp$ has partial groups for objects, and its morphisms are those functions that preserve pairwise commeasurability, identity, and products of commeasurable elements.

If $G$ is such a partial group, then we may extend the domain of its product to view it as a hyperoperation $\ostar \colon G \times G \to \P(G)$ by defining
\[
a \ostar b = \begin{cases}
a \ast b, & (a,b) \in \odot, \\
\varnothing, & (a,b) \notin \odot.
\end{cases}
\]
It is straightforward to verify that $(G, \ostar, e)$ becomes a commutative mosaic, and that every morphism of partial groups is a morphism of mosaics. In this way we obtain a faithful functor
\[
\begin{tikzcd}
\pGrp \ar[r, hookrightarrow] & \cMsc,
\end{tikzcd}
\]
and one can check that this is also full. In this way partial groups are isomorphic to a full subcategory of commutative mosaics.
\end{example}

Finally, we note a connection between mosaics, hypermagmas, and certain relational structures called crowds, defined by Lorscheid and Thas in~\cite[Section~5]{LorscheidThas}. 
A \emph{crowd} $G$ is a set along with a subset $R \subseteq G \times G \times G$ (called the \emph{crowd law}) and an \emph{identity} element $1 \in G$ satisfying the axioms:
\begin{itemize}
\item[(C1)] $(a,1,1) \in R$ if and only if $a = 1$,
\item[(C2)] $(a,b,1) \in R$ implies $(b,a,1) \in R$,
\item[(C3)] $(a,b,c) \in R$ implies $(c,a,b) \in R$.
\end{itemize}
With this, one defines the \emph{inverse} of an element $a \in G$ to be the subset
\[
a^{-1} = \{b \in G \mid (a,b,1) \in R\}.
\]
Axiom~(C2) guarantees that $b \in a^{-1}$ if and only if $a \in b^{-1}$. 
One can also define the \emph{product} of two elements $a,b \in G$ to be the subset
\begin{align*}
a \star_G b &= \{c \in G \mid (a,b,d) \in R \mbox{ for some } d \in c^{-1}\} \\
&= \{c \in G \mid (a,b,d) \in R \mbox{ and } c \in d^{-1}\}.
\end{align*}
A morphism of crowds is a function that preserves identity elements and crowd laws. We let $\Crowd$ denote the category of crowds and their morphisms. Then it is clear that the assignment $(G,R,1) \mapsto (G,\star_G)$ defines a functor
\[
\Crowd \to \HMag
\]
which acts identically on morphisms and therefore is faithful. 

It was shown in~\cite[Proposition~5.7]{LorscheidThas} that there is a fully faithful embedding of groups into crowds. This can be extended to an embedding of mosaics into crowds; we thank Oliver Lorscheid for first informing us about this functor.
If $M$ is a mosaic with elements $a,b,c \in M$, then reversibility and Lemma~\ref{lem:inversion} give
\[
1 \in (ab)c \iff c^{-1} \in ab \iff a^{-1} \in bc \iff 1 \in a(bc).
\]
Thus we can define a crowd law $R_M \subseteq M^3$ by
\begin{align*}
R_M &= \{(a,b,c) \in M^3 \mid 1 \in (ab)c\} \\
&= \{(a,b,c) \in M^3 \mid 1 \in a(bc)\}.
\end{align*}
The axioms (C1)--(C3) hold by uniqueness of inverses and the above observations. The assignment $(M,\star,1) \mapsto (M, R_M, 1)$ yields a functor
\[
\Msc \to \Crowd
\]
that again acts identically on morphisms and therefore is faithful. We will show in Proposition~\ref{prop:crowds} below that this is in fact fully faithful.

To better understand the relationship between these categories, we recall two further axioms defined in~\cite{LorscheidThas} that a crowd may satisfy:
\begin{itemize}
\item[(C4)] $a^{-1} \neq \varnothing$ for all $a \in G$,
\item[(C5)] If $(a,b,c) \in R$ and $a' \in a^{-1}$, $b' \in b^{-1}$, and $c' \in c^{-1}$, then $(c', b', a') \in R$.
\end{itemize}
The crowd defined from a mosaic has singleton inverses and thus satisfies~(C4). 
We let $\Crowd_4$ denote the full subcategory of $\Crowd$ consisting of those crowds satisfying axiom~(C4). We also let $\wHMag$ denote the category whose objects $(M,1)$ are weakly unital hypermagmas $M$ with a choice of weak unit $1 \in M$, and whose morphisms are those hypermagma morphisms that preserve the fixed weak units. This leads to the following.

\begin{proposition}\label{prop:crowds}
The assignments $(M,\star,1) \mapsto (M,R_M,1)$ and $(G,R,1) \mapsto (G, \star_G,1)$ define fully faithful functors
\[
\Msc \to \Crowd_4 \to \wHMag,
\]
where the first functor gives an isomorphism onto the full subcategory of crowds whose inverse sets are singletons and which satisfy axiom~(C5).
\end{proposition}

\begin{proof}
As explained above, we have faithful functors $\Msc \to \Crowd \to \HMag$ that act identically on morphisms, whose composite is the forgetful functor from mosaics to hypermagmas. We have also seen that the image of the first functor lies in $\Crowd_4$. To see that the second functor induces $\Crowd_4 \to \wHMag$ as described, first note that for any crowd $(G,R,1)$ and any element $a \in G$, we may use axioms~(C2) and~(C3) to find
\begin{align*}
1 \star_G a &= \{c \in G \mid (1,a,d) \in R \mbox{ and } c \in d^{-1}\} \\
&= \{c \in G \mid (a,d,1) \in R \mbox{ and } c \in d^{-1}\} \\ 
&= \{c \in G \mid d \in a^{-1} \mbox{ and } c \in d^{-1}\} \\
&= \bigcup_{d \in a^{-1}} d^{-1} =: (a^{-1})^{-1}.
\end{align*}
If $G$ satisfies~(C4), then there exists some $d \in a^{-1}$, so that $a \in d^{-1} \subseteq 1 \star_G a$, and a similar argument gives $a \in a \star_G 1$. This means that $1$ is a weak identity of $(G, \star_G)$, providing the functor $\Crowd_4 \to \wHMag$ as claimed.

The composite of the two functors in the statement is the forgetful functor $\Msc \to \wHMag$ which is fully faithful. Thus, if we prove that the second functor $\Crowd_4 \to \wHMag$ is full then it must also be the case that the first functor is full. 
To this end, suppose that $G$ and $H$ are crowds satisfying~(C4) and that $f \colon G \to H$ is a function which yields a morphism of hypermagmas $(G, \star_G,1_G) \to (H, \star_H,1_H)$ that preserves weak units. Using axiom~(C4) in the first implication below, we have
\begin{align*}
(a,b,c) \in R_G &\implies c' \in a \star_G b \mbox{ and } 1_G \in c' \star_G c \mbox{ for some } c' \in c^{-1} \\
&\implies f(c') \in f(a) \star_H f(b) \mbox{ and } 1_H = f(1_G) \in f(c') \star_H f(c) \\
&\implies (f(a), f(b), f(c)) \in R_H.
\end{align*}
So $f$ is a morphism of crowds, proving fullness. 

Finally we prove the claim about the image of mosaics within crowds. If $(M,\star)$ is a mosaic then the crowd $(M, R_M)$ clearly has singleton inverse sets, and reversibility implies that it satisfies (C5). Conversely, suppose that $(G,R,1)$ is a crowd with singleton inverses that satisfies~(C5). Then $(G, \star_G,1)$ is a weakly unital hypermagma. The uniqueness of crowd inverses gives
\[
1 \star_G a = (a^{-1})^{-1} = \{a\}
\]
and similarly $a \star_G 1 = \{a\}$ for all $a \in G$. Thus $1$ is a strict unit, and the crowd inverse also gives an inverse operation for the hypermagma structure. It remains to verify reversibility. Given $x,y,z \in G$, we use uniqueness of crowd inverses along with axioms~(C2) and~(C5) to obtain
\begin{align*}
x \in y \star_G z &\implies (y,z,x^{-1}) \in R \\
&\implies (x,z^{-1},y^{-1}) = ((x^{-1})^{-1}, z^{-1}, y^{-1}) \in R \\
&\implies y = (y^{-1})^{-1} \in x \star_G z^{-1},
\end{align*}
and a similar argument gives $z \in y^{-1} \star_G x$. So $(G, \star_G)$ is a mosaic as desired.
\end{proof}

\section{Categories of hypermagmas}

We now turn our attention to fundamental properties of the various categories of hyperstructures.

\subsection{Forgetful functors and (co)limits}
\label{sub:limits and colimits}

Although hyperoperations can be viewed as a type of multivalued function, morphisms of hypermagmas are ordinary functions between sets. Thus we have a forgetful functor from $\HMag$ to $\Set$, which will give us a good handle on formation of limits and colimits.

\begin{proposition}\label{prop:free and cofree}
The forgetful functor 
\[
U \colon \HMag \to \Set
\]
has both a left adjoint $F$ and a right adjoint $D$.
\end{proposition}

\begin{proof}
For any set $X$, define hyperoperations $\star_F$ and $\star_D$ by setting 
\begin{align*}
x\star_Fy &= \varnothing \mbox{ and} \\
x\star_Dy &= X
\end{align*}
for all $x,y \in X$. We obtain functors $F,D \colon \Set \to \HMag$ that act on objects setting $F(X) = (X, \star_F)$ and $D(X) = (X, \star_D)$ and act identically on morphisms. Note that 
\[
UF=UD=1_{\Set}. 
\]
For any hypermagma $G$, it follows directly from the triviality of $\star_F$ and $\star_D$ that $U$ induces natural bijections
\begin{align*}
\HMag(F(X), G) &\cong \Set(UF(X), U(G)) = \Set(X, U(G)), \\
\HMag(G, D(X)) &\cong \Set(U(G), UD(X)) = \Set(U(G), X).
\end{align*}
Thus we have an adjoint triple $F\dashv U\dashv D$.
\end{proof}

The functors $F$ and $D$ above respectively yield \emph{free hypermagmas} and \emph{cofree hypermagmas} over a given set.
One consequence of this will be discussed in Proposition~\ref{prop:mono and epi}, that monomorphisms and epimorphisms correspond to injective and surjective morphisms, respectively.
Notice, however,  that bijective morphisms in $\HMag$ need not be isomorphisms. For instance, given a set $X$ the identity function on underlying sets gives a bijective morphism $F(X) \to D(X)$ between the free and cofree hypermagmas on $X$. But these are certainly not isomorphic as hypermagmas if $X \neq \varnothing$. Thus $U$ does not reflect isomorphisms; in particular, the adjunction $F \dashv U$ is not monadic.
On the other hand, $D$ makes $\Set$ monadic over $\HMag$. A more natural way to state this is the following.

\begin{corollary}
The cofree functor $D \colon \Set \to \HMag$ gives an isomorphism of $\Set$ onto a reflective subcategory of $\HMag$.
\end{corollary}

\begin{proof}
We continue to let $U \colon \HMag \to \Set$ denote the forgetful functor. From the proof of Proposition~\ref{prop:free and cofree} we have $UD = 1_{\Set}$, so that the counit $\epsilon \colon UD \Rightarrow 1_{\Set}$ of the adjunction $U \dashv D$ is an isomorphism. In this case it is well known~\cite[Lemma~4.5.13]{Riehl:context} that $U$ is an equivalence onto a reflective subcategory. The fact that $U$ is an isomorphism onto its image again follows from the very strong condtion $UD = 1_{\Set}$.
\end{proof}

All small limits and colimits of hypermagmas can be constructed by endowing the corresponding (co)limit of sets with an appropriate hyperoperation. 

\begin{proposition}\label{prop:forgetful creates}
The forgetful functor $U \colon \HMag \to \Set$ lifts all limits and colimits. In particular, $\HMag$ is complete and cocomplete.
\end{proposition}

\begin{proof}
By the product-equalizer formulation of limits~\cite[V.2]{MacLane} and its dual, it suffices to show that $U$ lifts (co)products and (co)equalizers.

To verify the claim for (co)products, let $(G_i)_{i \in I}$ be a tuple of hypermagmas. Define a hyperoperation on the set-theoretic product $\prod G_i$ for elements $a = (a_i)_{i \in I}$ and $b = (b_i)_{i \in I}$ of $\prod G_i$ by setting
\[
a \star b = \prod (a_i \star b_i) \subseteq \prod G_i.
\]
It is routine to check that it also satisfies the universal property of a product in $\HMag$.
Similarly, we may endow  the disjoint union $\bigsqcup G_i$ with the hyperoperation 
\[
\sqcup_{i,j} G_i \times G_j \cong \left( \sqcup G_i \right) \times \left( \sqcup G_j \right)
\to \P\left( \sqcup G_i\right) 
\]
that is empty on each $G_i \times G_j$ for $i \neq j$ but agrees with the hyperoperation of $G_i$ on $G_i \times G_i \to \P(G_i) \subseteq \P(\sqcup G_i)$. One can then verify that this is a coproduct of hypermagmas.

Next let $f,g \colon G \to H$ be a parallel pair of morphisms in $\HMag$, and denote the set-theoretic equalizer and coequalizer by
\[
\begin{tikzcd}
E \arrow[r, hookrightarrow, "i"] & G \arrow[r, shift left=1, "f"] \arrow[r, shift right=1, swap, "g"] & H \arrow[r, twoheadrightarrow, "\pi"] & K.
\end{tikzcd}
\] 
Define a hyperoperation on $E = \{x \in G \mid f(x) = g(x)\}$ by 
\[
x \star_E y = (x \star_G y) \cap E
\]
for any $x,y \in E$. It is straightforward to verify that $(E,\star_E)$ acts as an equalizer of $f$ and $g$ in the category of hypermagmas. 

Finally, we define a hyperoperation on $K$ by setting
\[
x \star_K y = \pi(\pi^{-1}(x) \star_H \pi^{-1}(y))
\]
for $x,y \in K$. This equips $K$ with the structure of a hypermagma, and the definition of $\star_K$ ensures that $\pi \colon H \to K$ is a morphism of hypermagmas.
Suppose that $w \in \HMag(H,L)$ coequalizes $f$ and $g$, and let 
\[
w \colon H \overset{\pi}{\longrightarrow} K \overset{u}{\longrightarrow} L
\]
be the factorization of $w$ in $\Set$ given by the universal property of $K$. To verify that $u$ is a morphism in $\HMag$, let $x,y \in K$ and fix $z \in x \star_K y$. Then there exists $z_0 \in x_0 \star y_0$ for some $x_0 \in \pi^{-1}(x)$ and $y_0 \in \pi^{-1}(y)$ such that $z = \pi(z_0)$. Then
\begin{align*}
u(z) &= u(\pi(z_0)) = w(z_0) \\
&\in w(x_0 \star y_0) \subseteq w(x_0) \star w(y_0) \\
&= u(\pi(x_0)) \star u(\pi(y_0)) = u(x) \star u(y).
\end{align*}
This establishes $u(x \star y) \subseteq u(x) \star u(y)$, so that $u$ is a morphism of hypermagmas as desired.
\end{proof}

In the special case of the empty (co)limit, it follows that $U$ also lifts initial and terminal objects. 
Indeed, it is clear that $\HMag$ has initial object given by the empty set with its unique hyperoperation $\varnothing \times \varnothing \to \P(\varnothing)$ and terminal object given by the singleton set $1$ with the hyperoperation corresponding to its unique binary operation $1 \times 1 \to 1$.

\separate

In the case of unital hypermagmas, there is a ``finer'' forgetful functor to the category of \emph{pointed} sets 
\begin{equation}\label{eq:pointed forgetful}
\begin{tikzcd}
 \uHMag \ar[r, "U_\bullet"] \ar[dr, swap, "U"] & \Set_\bullet \ar[d, dashed]\\
 & \Set
\end{tikzcd}
\end{equation}
that can be more useful than the forgetful functor to sets. It is given by sending a unital hypermagma $M$ to the pointed set $(M,e_M)$. 

The category $\Set_\bullet$ is complete and cocomplete, with products given by the direct product and coproducts by the \emph{wedge sum}: the quotient of the disjoint union that identifies all basepoints.

\begin{theorem}\label{thm:pointed forgetful}
The forgetful functors from unital hypermagmas to each of sets and pointed sets in~\eqref{eq:pointed forgetful}
both have left adjoints. Both functors lift all limits in $\uHMag$, and $U_\bullet$ lifts all coproducts in $\uHMag$.
\end{theorem}

\begin{proof}
Given $(X,x_0) \in \Set_\bullet$, define a hyperoperation $\star_F$ on $X$ by taking $x_0 \star_F x = x = x \star_F x_0$ for all $x \in X$ and $x \star_F y = \varnothing$ for all $x,y \in X \setminus \{x_0\}$. We obtain a functor
\begin{align*}
F \colon \Set_\bullet &\to \uHMag, \\
(X,x_0) &\mapsto (X, \star_F, x_0).
\end{align*}
Notice that $U_\bullet F = 1_{\Set_\bullet}$, and that $U_\bullet$ induces natural bijections
\[
\uHMag(F(X,x_0), M) \cong \Set_\bullet(U_\bullet F(X), U_\bullet(M)) = \Set_\bullet((X,x_0), U_\bullet(M)).
\]
Thus $F \dashv U_\bullet$ as desired. Note that the forgetful functor $\Set_\bullet \dashrightarrow \Set$ also has a left adjoint (given by freely adjoining a basepoint), which composes with $F$ to yield a left adjoint to $U \colon \uHMag \to \Set$.

To see that these forgetful functors lift products, fix a family $(M_i, \star_i, e_i)_{i \in I}$ of unital hypermagmas.
Because the product hypermagma $\prod M_i$ of Proposition~\ref{prop:forgetful creates} has identity element $e = (e_i) \in \prod M_i$ and the projections to each $M_i$ are unital, this also forms a product in the category $\uHMag$.

Next take the wedge sum $S = \bigvee (M_i, e_i)$, and again denote its basepoint by $e \in S$. Let $I_0 = I \sqcup \{0\}$ be a poset in which $0$ is the smallest element and all elements of $I$ are incomparable. Then let $D \colon I_0 \to \HMag$ be the diagram given on objects by $D(0) = 1$ and $D(i) = M_i$ for $i \in I$, and which sends each arrow $0 \leq i$ to the unique unital morphism $1 \to M_i$. 
Then the underlying set of the colimit of $D$ in $\HMag$ coincides with $S = \bigvee M_i$ and unit given by the basepoint $e \in S$. One can verify from its construction via the diagram $D$ that it satisfies the universal property of the coproduct in $\uHMag$.

Finally, the lifting of the equalizer of a parallel pair of morphisms $f,g \in \uHMag(G,H)$ is deduced just as in the proof of Proposition~\ref{prop:forgetful creates}. Indeed, the set-theoretic equalizers of $G$ and $H$ will be basepoint-preserving because $f$ and $g$ preserve the identity elements, and the hyperoperations constructed as in that proof will be unital for the same reason. 
\end{proof}

Unlike the case of hypermagmas, the forgetful functors of~\eqref{eq:pointed forgetful} do not have right adjoints. Indeed, if they had right adjoints then the forgetful functors $U$ and $U_\bullet$ would both preserve coequalizers. But the following counterexample illustrates a case where this fails to happen.

\begin{example}\label{ex:not unital}
Endow $D = \{0,1,2\}$ with the structure of a commutative hypermagma where $0$ is an identity and 
\[
1+1 = 1+2 = 2+2 = D.
\]
On the underlying pointed set $X = U_\bullet D$, we may also define the structure of a hypermagma using the free unital hypermagma $FX$ of Theorem~\ref{thm:pointed forgetful}. Let $f, g \in \uHMag(FX, D)$ be the morphisms defined by
\[
f(1) = 1, \ f(2) = 2, \quad g(1) = 0, \ g(2) = 2.
\] 
Then the coequalizer $K$ of $f$ and $g$ in either of the categories $\Set_\bullet$ or $\HMag$ has underlying set given by the set-theoretic coequalizer, which is equal to $\{[0],[2]\}$. Furthermore, the coequalizer in $\HMag$ has a hyperoperation satisfying
\begin{align*}
[0] + [0] &= \{[t] \mid t \in x+y \mbox{ where } [x] = [y]=[0] \mbox{ for } x,y \in D\} \\
&= \{[t] \mid t \in (0 + 0) \cup (0 + 1) \cup (1 + 1) \} \\
&= \{[0], [2]\}.
\end{align*}
Thus $[0]$ is not an identity, so that $K$ is not a unital hypermagma. This means that the forgetful functor $\uHMag \to \HMag$ (and by extension $\uHMag \to \Set$) as well as $\uHMag \to \Set_\bullet$ do not preserve coequalizers. 
\end{example}

To overcome this problem, we will need to make use of the following universal construction. If $f \colon M \to N$ is a morphism of hypermagmas and if $N$ has identity $e$, we define the \emph{kernel} of $f$ to be 
\[
\ker f = f^{-1}(e_N) \subseteq M.
\]
Note that this term will apply even if $M$ itself is not unital, or if $M$ is unital but $f$ does not preserve the identity. In such cases, it is possible for the kernel of a morphism to be empty.

\begin{definition}\label{def:unitization}
Given a hypermagma $M$ and a (possibly empty) subset $E \subseteq M$, a \emph{unitization of $M$ relative to $E$} $(M_E, \pi_E)$ is a unital hypermagma 
equipped with a universal morphism $\pi_E \colon M \to M_E$ that sends every element of $E$ to the unit of $M_E$. That is, the pair $(M_E, \pi_E)$ represents the functor $\uHMag \to \Set$ given by
\[
N \mapsto \{f \in \HMag(M,N) \mid E \subseteq \ker f = f^{-1}(e_N)\}.
\]
\end{definition}

We will show that this universal object exists below. Its construction will be facilitated by the following property.

\begin{definition}\label{def:absorptive}
For a hypermagma $M$, a subset $K \subseteq M$ is \emph{absorptive} if, for all $x \in M$,
\[
(x \star K \cup K \star x) \cap K \neq \varnothing \implies x \in K.
\]
\end{definition}

\begin{remark}\label{rem:absorptive}
Several comments are in order. Let $M$ denote a hypermagma throughout.
\begin{enumerate}
\item If $f \in \HMag(M,N)$ and $N$ is unital, then $\ker f \subseteq M$ is absorptive. Indeed, let $x \in M$ and suppose there exists $y \in x \star \ker f \cap \ker f$ (with a symmetric argument applying in case $y \in (\ker f) \star x \cap \ker f$). Then there exists $u \in \ker f$ with $y \in x \star u$, so that
\[
e_N = f(y) \subseteq f(x \star u) \subseteq f(x) \star f(u) = f(x) \star e_N = \{f(x)\}.
\]
This forces $f(x) = e_N$, so that $x \in \ker f$.
\item If $M$ is unital with $e = e_M$ and $K \subseteq M$ is absorptive and nonempty, then $e \in K$. Indeed, if $k \in K$ then $e \star k = k \in (e \star K) \cap K$ yields $e \in K$.
\item The absorptive property can be seen as dual to the property of being a strict subhypermagma in the following way. Suppose $K \subseteq M$ and that $x,y,z \in M$ satisfy 
\[
z \in x \star y \cup y \star x. 
\]
If $K$ is a strict subhypermagma then $x,y \in K \implies z \in K$. On the other hand, if $K$ is absorptive then $y,z \in K \implies x \in K$.
\item The properties of being either absorptive or a strict subhypermagma are preserved under arbitrary intersection in $\P(M)$. Thus every subset of $M$ is contained in a smallest (absorptive) strict subhypermagma of $M$, which we say is \emph{generated} by that set. 
\item If $M$ is a mosaic and $K \subseteq M$ is a strict submosaic, then $K$ is absorptive. Indeed, if $x \in M$ and $(x \star K \cup K \star x) \cap K \neq \varnothing$, we may assume without loss of generality that there exist $y,z \in K$ with $z \in x \star y$. But then $y^{-1} \in K$ and strictness of $K$ yield $x \in z \star y^{-1} \subseteq K$ as desired.
\item If $M$ is a mosaic and $K \subseteq M$ is absorptive and nonempty, then $K$ is closed under inverses. Indeed, suppose $y \in K$. Note that $e = e_M \in K$ by~(2), and also $e \in y^{-1} \star y \subseteq y^{-1} \star K$. By the absorptive property, we find that $y^{-1} \in K$. 
\end{enumerate}
\end{remark}

\begin{lemma}\label{lem:unitization}
Let $M$ be a hypermagma and let $E \subseteq M$ be a subset. Then there exists a unitization $(M_E, \pi_E)$, whose kernel is the smallest absorptive strict subhypermagma of $M$ containing $E$.
Furthermore, $\pi_E \colon M \to M_E$ is short if $E$ satisfies the following condition: for all $x \in M$, $x \star E \neq \varnothing \neq E \star x$.
\end{lemma}

\begin{proof}
If $E = \varnothing$ we can easily extend the hyperoperation of $M$ to one on $M_\varnothing := M \sqcup \{e\}$ such that $e$ is a unit. It is then clear that this object with the natural inclusion $\pi_\varnothing = i \colon M \hookrightarrow M_\varnothing$ satisfies the universal property. 

So we may assume now that $E \neq \varnothing$. Let $K \subseteq M$ be the smallest absorptive strict subhypermagma of $M$ that contains $E$, as in Remark~\ref{rem:absorptive}(4). Let $\sim$ be the equivalence relation on $M$ defined by setting $x \sim y$ if and only if either $x,y \in K$ or there exists a sequence $x = z_1, z_2, \dots, z_n = y$ in $M$ such that each
\[
z_{i+1} \in z_i \star K \cup K \star z_i \quad \mbox{or} \quad z_i \in z_{i+1} \star K \cup K \star z_{i+1}
\]
for $i = 1, \dots, n-1$.
This is evidently reflexive (by taking $n = 1$ above) and symmetric.
It follows from the fact that $K$ is a strict absorptive subhypermagma that if $x \in M$ and $u \in K$, then $x \sim u \iff x \in K$.
From this one can verify that $\sim$ is transitive and that $K$ forms an equivalence class.

Let $M_E$ be the quotient of $M$ by $\sim$ and let $\pi_E = \pi$ denote the canonical surjection
\[
\pi \colon M \twoheadrightarrow M/{\sim} = M_E.
\]
Let $e \in M_E$ denote the image of the equivalence class $K$, so that $\pi^{-1}(e) = K \supseteq E$. 
We will alternatively denote equivalence classes of $x \in M$ as $[x] = \pi(x) \in M_E$. 
Define a hyperoperation on $M_E$ by
\[
[x] \star [y] :=
\begin{cases}
\pi(\pi^{-1}([x]) \star \pi^{-1}([y])), & [x] \neq e \neq [y], \\
\{[x]\}, & [y] = e, \\
\{[y]\}, & [x] = e.
\end{cases}
\]
By construction $e$ is an identity for $M_E$. We claim that $\pi \in \HMag(M, M_E)$ is a morphism of hypermagmas, which is to say that $\pi(x \star y) \subseteq \pi(x) \star \pi(y)$ for all $x, y \in M$. This follows immediately from the construction above if $[x] \neq e \neq [y]$. 
If $[y]=e$ then $y \sim u$ for some $u \in E$. Then if $z \in x \star y$, by construction of $\sim$ it follows that $z \sim x$. Thus $\pi(x \star y) \subseteq [x] = [x] \star e = \pi(x) \star \pi(y)$, and a symmetric argument holds in case $[x] = e$.

It remains to demonstrate the universal property. Suppose that $N$ is a unital hypermagma and that $f \in \HMag(M,N)$ satisfies $f(E) = e_N$. Let $\approx$ be the equivalence relation on $M$ given by 
\[
x \approx y \iff f(x) = f(y).
\]
Because $f(E) = e_N$, it follows from Remark~\ref{rem:absorptive}(1) that $K \subseteq \ker f$. Thus we have $K \times K \subseteq {\approx}$. Furthermore, suppose that $x,y \in M$ are such that $y \in x \star K \cup K \star x$. Fix $u \in K$ with $y \in x \star u \cup u \star x$. Then  
\begin{align*}
f(y) &\in f(x \star u) \cup f(u \star x) \\
	&\subseteq f(x) \star f(u) \cup f(u) \star f(x) \\
	&= f(x) \star e_N \cup e_N \star f(x) \\
	&= \{f(x)\}.
\end{align*}
Then $f(y) = f(x)$ so that $y \approx x$. It follows by construction of $\sim$ that ${\sim} \subseteq {\approx}$. 

Thus $f$ factors uniquely as $f = g \circ \pi$ where $g \colon M_E \to N$ is the well-defined function given by $g([x]) = f(x)$. We claim that $g \in \uHMag(M_E,N)$, from which the universal property of $(M_E, \pi)$ will follow. By construction we have $g(e) = f(\pi(E)) = e_N$, so that $g$ preserves unit elements. 
If $[x] \neq e \neq [y]$ in $M_E$, then
\begin{align*}
g([x] \star [y]) &= g(\pi(\pi^{-1}(x) \star \pi^{-1}(y))) \\
	&= f(\pi^{-1}(x) \star \pi^{-1}(y))) \\
	&\subseteq f(\pi^{-1}(x)) \star f(\pi^{-1}(y)) \\
	&= g([x]) \star g([y]).
\end{align*}
An easier argument verifies that $g([x] \star [y]) \subseteq g([x]) \star g([y])$ in case either $[x]$ or $[y]$ equals $e \in M_E$.

Finally, we verify the claim about shortness of $\pi$. Note that 
\[
[x] \star [y] = \pi(\pi^{-1}([x]) \star \pi^{-1}([x]))
\]
holds for all nonidentity elements $[x] \neq e \neq [y]$  by construction. In the case where $x \star E \neq \varnothing \neq E \star x$ for all $x \in M$, we can verify in the case $[y] = e$ as follows: 
\[
\varnothing \neq \pi( x \star E) \subseteq \pi(\pi^{-1}([x]) \star \pi^{-1}(e)) \subseteq [x] \star e = \{[x]\},
\]
from which we conclude that $\pi(\pi^{-1}([x]) \star \pi^{-1}(e)) = \{[x]\} = [x] \star e$. A symmetric argument applies in the case where $[x] = e$ and $[y]$ is arbitrary.
\end{proof}

We can now apply unitization to construct coequalizers as follows.

\begin{theorem}\label{thm:unital cocomplete}
The category $\uHMag$ is complete and cocomplete.
\end{theorem}

\begin{proof}
In light of Theorem~\ref{thm:pointed forgetful}, it only remains to show that the category has coequalizers. Consider a parallel pair of morphisms $f,g \in \uHMag(M,N)$, form their coequalizer $(L, \pi_L)$ in $\HMag$, and denote $E = \{\pi_L(e_N)\} \subseteq L$. We claim that the morphism $\pi = \pi_E \circ \pi_L \colon N \to L_E$ in the diagram
\[
\begin{tikzcd}
M \ar[r, shift left=1, "f"] \ar[r, shift right =1, swap, "g"] 
	& N \ar[r, twoheadrightarrow, "\pi_L"] \ar[rr, dashed, twoheadrightarrow, bend right=8ex, "\pi"]
	& L \ar[r, twoheadrightarrow, "\pi_E"] 
	& L_E
\end{tikzcd}
\]
induced by the unitization $(L_E, \pi_E)$ of Lemma~\ref{lem:unitization} is a coequalizer in $\uHMag$. It follows from the construction of $\pi$ both that it is unital and that it coequalizes $f$ and $g$. Furthermore, given any other coequalizing morphism $h \in \uHMag(N,P) \subseteq \HMag(N,P)$, we see that $h$ factors uniquely through $\pi_L$ via a morphism of hypermagmas. But the fact that $h$ is unital means that it must further factor through $\pi_E$ by the universal property of $L_E$. Thus $h$ factors uniquely through $\pi = \pi_E \pi_L$ as desired. 
\end{proof}

\begin{remark}\label{rem:short coequalizer}
Note that the coequalizer $\pi = \pi_E \circ \pi_L$ constructed above is short. Indeed, $\pi_L$ is short by its construction in Proposition~\ref{prop:forgetful creates}, and because shortness is preserved under composition, it is enough to verify that $\pi_E$ is short. This will follow from Lemma~\ref{lem:unitization} if we can verify that, for all $x \in L$, we have $x \star E \neq \varnothing \neq E \star x$. Writing $x = \pi_L(y)$ for some $y \in N$, we have
\[
x \star E = \pi_L(y) \star \pi_L(e_N) \supseteq \pi_L(y \star e_N) = \{x\}.
\]
Thus $x \star E \neq \varnothing$ as desired, and similarly we have $x \in E \star x \neq \varnothing$. 
\end{remark}

\subsection{Characterizations of various morphisms}
\label{sub:epi and mono}

In this section we focus on characterizations of various morphisms in categories of hypermagmas. This includes characterizations of categorically-defined morphisms in terms of hypermagma structure, as well as characterizations of strict morphisms of hypermagmas in categorical terms. 

We begin by investigating various degrees monomorphisms and epimorphisms in the categories $\HMag$ and $\uHMag$, whose definitions we recall. There is no surprise in the characterization of ordinary monomorphisms and epimorphisms.

\begin{proposition}\label{prop:mono and epi}
In each of the categories $\HMag$ and $\uHMag$, the monomorphisms are the injective morphisms and the epimorphisms are the surjective morphisms.
\end{proposition}

\begin{proof}
Suppose that $\C$ is either of the categories $\HMag$ or $\uHMag$. Because the forgetful functor $U \colon \C \to \Set$ is faithful, it reflects monomorphisms and epimorphisms. The fact that $U$ has a left adjoint in the case of $\HMag$ further implies that it preserves monomorphisms. 

For $\C = \HMag$ the forgetful functor also has a right adjoint, so that it preserves epimorphisms as well. Finally, we show that if $f \in \uHMag(M,N)$ is not surjective, then it is not an epimorphism. Consider the unital hypermagma $D = \{1,-1,0\}$ for which $0$ is an (additive) identity and $x + y = D$ for $x,y \in D \setminus \{0\}$. Define morphisms $g,h \in \uHMag(N,D)$ on nonidentity elements $x \in N \setminus \{e_N\}$ by $g(x) = 1$ and 
\[
h(x) = \begin{cases}
\phantom{-}1, & x \in f(M), \\
-1, & x \notin f(M).
\end{cases}
\]
Then $g \circ f = h \circ f$, but $g \neq h$ since $f(M) \subsetneq N$. Thus $f$ is not epic.
\end{proof}

Next we consider regular monomorphisms and epimorphisms. Recall that a \emph{regular} epimorphism (resp., monomorphism) in a category $\C$ is defined to be a coequalizer (resp., equalizer) in $\C$. For hypermagmas, the (co)short morphisms of Definition~\ref{def:short} turn out to be the correct characterization.

\begin{theorem}\label{thm:regular epi and mono}
In each of the categories $\HMag$ and $\uHMag$, the regular epimorphisms are the short morphisms and the regular monomorphisms are the coshort morphisms.
\end{theorem}

\begin{proof}
It is clear from the construction of coequalizers in Proposition~\ref{prop:forgetful creates} that regular epimorphisms in $\HMag$ are short, and the case of $\uHMag$ was discussed in Remark~\ref{rem:short coequalizer}. 
Conversely, suppose that $p \colon M \twoheadrightarrow N$ is a short morphism of (unital) hypermagmas.
Let $\sim$ be the equivalence relation on $M$ defined by $x \sim y$ if and only if $p(x) = p(y)$, and let $\{x_\alpha\}$ be a complete system of representatives for these equivalence classes (including $e_M$ in the unital case). Let $M_0 = F(U(M))$ be the free (unital) hypermagma on the underlying (pointed) set of $M$, so that the product of any two (non-unit) elements is always empty. We define morphisms $f, g \colon M_0 \to M$ by $f(x) = x$ for all $x$ and $g(x) = x_\alpha$ for the unique $\alpha$ such that $x \sim x_\alpha$. Then $f$ and $g$ are morphisms of (unital) hypermagmas for which $p$ is the  coequalizer as (pointed) sets. 

To verify that $p$ is in fact the coequalizer of $f$ and $g$ in $\HMag$ (resp., $\uHMag$), suppose that $q \colon M \to L$ is a morphism of hypermagmas with $q \circ f = q \circ g$. Then $q$ factors uniquely as $q = h \circ p$ for some function $h \colon N \to L$ (which preserves basepoints in the unital case), and we wish to demonstrate that $h$ is a morphism of hypermagmas. Suppose that $x, y \in N$. Then because $p$ is a short morphism, we have
\begin{align*}
h(x \star y) &= h(p(p^{-1}(x) \star p^{-1}(y))) \\
&= q(p^{-1}(x) \star p^{-1}(y)) \\
&\subseteq q(p^{-1}(x)) \star q(p^{-1}(y)) \\
&= h(p(p^{-1}(x))) \star h(p(p^{-1}(y))) \\
&= h(x) \star h(y).
\end{align*}
So $h$ is a morphism of hypermagmas, proving that $p$ is a coequalizer in $\HMag$ (resp., $\uHMag$).

In the case of monomorphisms, it again follows from the proof of Proposition~\ref{prop:forgetful creates} and from Theorem~\ref{thm:pointed forgetful} that equalizers in $\HMag$ and $\uHMag$ are coshort. To verify the converse in the unital case, suppose that $i \in \uHMag(M,N)$ is a coshort morphism. Consider the commutative unital hypermagma $D = \{0, 1, 2\}$ of Example~\ref{ex:not unital}, where $0$ is the identity and $1+1 = 1+2 = 2 + 2 = D$. Define $f, g \in \uHMag(N,D)$ by 
\[
f(x) = \begin{cases}
0, &  x=e_N,\\
1, & x \neq e_N,
\end{cases}
\quad
g(x) = \begin{cases}
0, & x=e_N, \\
1, & x\in i(M) \setminus \{e_N\}, \\
2, & x \notin i(M).
\end{cases}
\]
Both of these are morphisms of unital hypermagmas, and we claim that $i$ is the equalizer of this pair. It is clear that $i$ is the set-theoretic equalizer of $f$ and $g$. So if $q \in \uHMag(L,N)$ satisfies $f \circ q = g \circ q$, then it uniquely factors as $q = i \circ h$ for some function $h \colon L \to M$. It remains to show that $h$ is a morphism of unital hypermagmas. 
Indeed, for any $x,y \in L$ we have 
\[
ih(x \star_L y) = q(x \star_L y) \subseteq q(x) \star_N q(y) = ih(x) \star_N ih(y)
\]
so it follows (invoking coshortness) that
\[
h(x \star_L y) \subseteq i^{-1}(ih(x\star_L y)) \subseteq i^{-1}(ih(x) \star_N ih(y)) \subseteq h(x) \star_M h(y)
\]
as desired.

A similar argument shows that if $i \in \HMag(M,N)$ is coshort, then it is an equalizer. In this case, one uses the cofree hypermagma $D = D(\{0,1\})$ and shows that $i$ equalizes the constant~$0$ morphism and a non-constant morphism. We omit the proof for the sake of brevity.
\end{proof}

We remark that several natural properties of hypermagmas are not preserved under epimorphic images. For instance, free (unital) hypermagmas are both commutative and associative for trivial reasons. Since every (unital) hypermagma admits a surjective morphism from the free object on its underlying set, the existence of noncommutative and nonassociative structures shows that these properties are not preserved. 
By contrast, these properties are preserved under \emph{regular} epimorphic images, as shown in Lemma~\ref{lem:short image}.

We can also use this characterization to show that (unital) hypermagmas form a regular category, which is a fundamental property in categorical algebra~\cite{Gran:regular}.
A category $\C$ is \emph{regular} if it is finitely complete, coequalizers of kernel pairs exist in $\C$, and regular epimorphisms are stable under pullback in $\C$. 
In any regular category, each morphism factors~\cite[Theorem~1.11]{Gran:regular} uniquely up to isomorphism as a regular epimorphism followed by a monomorphism. In this way, morphisms in regular categories have images that are pullback-stable.

\begin{corollary}\label{cor:regular}
The categories $\HMag$ and $\uHMag$ are both regular. 
\end{corollary}

\begin{proof}
Each of these categories is (finitely) complete and admits all coequalizers. Thus it remains to show that regular epimorphisms---or equivalently, short morphisms---are preserved under pullbacks. Let $p \colon M \to N$ be a short surjective morphism in either of these categories, and let $f \colon L \to N$ be any morphism. The pullback
\[
\begin{tikzcd}
L \times_N M \ar[r, "\pi"] \ar[d] & L \ar[d, "f"] \\
M \ar[r, "p"] & N
\end{tikzcd}
\]
has underlying set given by the pullback of sets $L \times_N M = \{(x,x') \in L \times M \mid f(x) = p(x')\}$, with hyperoperation
\[
(x,x') \star (y, y') = [(x \star_L y) \times (x' \star_M y')] \cap L \times_N M.
\]
Since $p$ is surejctive and the diagram is a pullback in $\Set$, the projection $\pi$ is also surjective. 
We need to show that 
\[
x \star y \subset \pi(\pi^{-1}(x)\star \pi^{-1}(y))
\] 
for all $x, y\in L$. To prove this, fix $z \in x \star y$. 
By surjectivity of $\pi$, there exist $x', y' \in M$ with $(x,x') \in \pi^{-1}(x)$ and $(y,y') \in \pi^{-1}(y)$.
This means that $f(x) = p(x')$ and $f(y) = p(y')$. From the fact that $p$ is short, we obtain
\[
f(x \star y) \subseteq f(x) \star f(y) = p(p^{-1}(f(x)) \star p^{-1}(f(y))).
\]
Since $z \in x \star y$, there exist $x'',y'' \in M$ such that $p(x'') = f(x)$, $p(y'') = f(y)$, and $f(z) \in p(x'' \star y'')$, which in turn means that there exists $z' \in x'' \star y''$ with $f(z) = p(z'')$. Thus we have $(x,x''),(y,y''), (z,z') \in L \times_N M$ with 
\[
(z,z') \in [(x \star_L y) \times (x'' \star_M y'')] \cap L \times_N M = (x,x'') \star (y,y'').
\]
Then
\[
z = \pi(z,z') \in \pi((x,x'') \star (y,y'')) \subseteq \pi(\pi^{-1}(x) \star \pi^{-1}(y))
\]
as desired.
\end{proof}

Now we consider normal monomorphisms and epimorphisms. Recall that if $\C$ is a category with a zero object, then a \emph{normal} epimorphism (resp., monomorphism) is defined to be a coequalizer (resp., equalizer) between any morphism and a zero morphism, or in other words, a cokernel (resp., kernel) of any morphism.

\begin{theorem}\label{thm:normal}
In the category $\uHMag$, the normal epimorphisms are the unitizations at nonempty subsets, and the normal monomorphisms correspond to the absorptive strict unital subhypermagmas.
\end{theorem}

\begin{proof}
Let $E \subseteq M$ be a nonempty subset.
To see that any unitization $\pi_E \colon M \twoheadrightarrow M_E$ is normal, first note that we may replace $E$ with the kernel of $\pi_E$ to assume without loss of generality that $E \neq \varnothing$ is a strict absorptive subhypermagma of $M$. By Remark~\ref{rem:absorptive}, we have $e_M \in E$, so that $E$ is a unital subhypermagma. Then it follows from the universal property of the unitization that $\pi_E$ is the coequalizer of $e_M,i_E \colon E \to M$.
Conversely, if $f \in \uHMag(L,M)$, then it is straightforward to see that the coequalizer of $f$ and the trivial morphism $e_M \colon L \to M$ is the same as the unitization $\pi_E \colon M \to M_E$ for $E = f(L) \subseteq M$.

Next, for any $g \in \uHMag(M,N)$, the equalizer of $g$ and $e_N$ is the inclusion of the kernel $i \colon K \hookrightarrow M$ where $K = f^{-1}(e_N)$. From Remark~\ref{rem:absorptive}(1) we see that $K$ is an absorptive strict unital subhypermagma. 
Conversely, if $K \subseteq M$ is an absorptive strict unital subhypermagma, then it follows from Lemma~\ref{lem:unitization} that $K$ is the kernel of the unitization $\pi \colon M \to M_K$, so that the inclusion $K \hookrightarrow M$ is a normal monomorphism.
\end{proof}

In Section~\ref{sec:mosaics} we will consider (co)limits in the category of mosaics. At that point we will establish similar characterizations of the various epimorphisms and monomorphisms in $\Msc$ and $\cMsc$.

\separate

The previous theorem can also be viewed as a categorical characterization of strict injective morphisms of hypermagmas. Given the importance of strict morphisms in the literature on hypergroups and hyperrings, it seems natural to ask whether there is a categorical characterization of strict morphisms in general. We answer this question in Theorem~\ref{thm:strict morphisms} below.

Let $\C$ denote a category of hypermagmas, i.e.~a category with a faithful forgetful functor to $\HMag$. We have the following functor 
\begin{align*}
E \colon \C &\to \Set \\
M &\mapsto \{(x,y,z) \in M^3 \mid z \in x \star y\}.
\end{align*}
If this functor happens to be representable, we let $\mathcal{E}_\C$ denote the representing object in $\C$. 

\begin{lemma}\label{lem:representing object}
If $\C$ is any of the categories $\HMag$, $\uHMag$, $\Msc$, or $\cMsc$, then the object $\E_\C$ defined above exists in $\C$.
\end{lemma}

\begin{proof}
In each case we define the ``freest'' possible object $\E$ generated by elements $a,b,c \in \E$, subject to the condition
\[
c \in a \star b. 
\]We will define the objects and omit the straightforward proofs that they represent the corresponding functor.

If $\C = \HMag$, we take $\E_\C = \{a,b,c\}$ with hyperoperation
\[
x \star y = \begin{cases}
c, & x = a \mbox{ and } y = b, \\
\varnothing, & \mbox{otherwise}.
\end{cases}
\]
If $\C = \uHMag$ we take $\E_\C = \{e,a,b,c\}$ with the hyperoperation extending the definition above so that $e$ becomes an identity element. (In other words, it is the unitization of $\E_{\HMag}$ at the empty set.)

In the case $\C = \Msc$, we define $\E_\C = \{e, a^{\pm 1}, b^{\pm 1}, c^{\pm 1}\}$ with the hyperoperation such that $e$ is an identity and whose products of nonidentity elements are given by the following table:
\[
\begin{array}{|c|c|c | c | c | c | c|} 
  \hline
   \star& a & a^{-1} & b & b^{-1} & c & c^{-1} \\ 
  \hline
  a & \varnothing & e & c & \varnothing & \varnothing & \varnothing\\ 
  \hline
  a^{-1} & e & \varnothing & \varnothing & \varnothing & b & \varnothing\\ 
  \hline
  b & \varnothing & \varnothing & \varnothing & e & \varnothing & a^{-1} \\
  \hline
  b^{-1} & \varnothing & c^{-1} & e & \ \varnothing & \varnothing & \varnothing\\
  \hline
  c & \varnothing & \varnothing & \varnothing & a & \varnothing & e\\
  \hline
  c^{-1} & b^{-1} & \varnothing & \varnothing & \varnothing & e & \varnothing\\
  \hline  
  \end{array}
\]

In the case $\C = \cMsc$, we symmetrize the table above by setting $\E_\C = \{0, \pm a, \pm b, \pm c\}$ and defining hyperaddition of nonzero elements by the table:
\[
\begin{array}{ | c | c | c | c | c | c | c |} 
  \hline
   + & a & -a & b & -b & c & -c\\ 
  \hline
  a & \varnothing & 0 & c & \varnothing & \varnothing & -b\\ 
  \hline
  -a & 0 & \varnothing & \varnothing & -c & b & \varnothing\\ 
  \hline
  b & c & \varnothing & \varnothing & 0 & \varnothing & -a\\
  \hline
  -b & \varnothing & -c & 0 & \ \varnothing & a & \varnothing\\
  \hline
  c & \varnothing & b & \varnothing & a & \varnothing & 0\\
  \hline
  -c & -b & \varnothing & -a & \varnothing & 0 & \varnothing\\
  \hline  
  \end{array}
\]

The reader can verify from construction of each $\E = \E_\C$ that we have bijections
\begin{align*}
\C(\E, M) &\cong E(M) = \{(x,y,z) \in M^3 \mid z \in x \star y\}, \\
f &\mapsto (f(a), f(b), f(c)),
\end{align*}
with $\id_\E$ corresponding to the universal element $(a,b,c) \in E(\E) \subseteq \E^3$. 
For instance, in the case $\C = \cMsc$, let $(M,+,0)$ be a commutative mosaic. 
For a morphism $f \colon \E_{\cMsc} \to M$, the fact that
\[
f(c) = f(a + b) \subseteq f(a) + f(b).
\]
shows that $(f(a), f(b), f(c)) \in E(M)$. Conversely, for any $(x,y,z) \in E(M)$, define a function $f \colon \E_{\cMsc} \to M$ by setting $f(\pm a) = \pm x$, $f(\pm b) = \pm y$, and $f(\pm c) = \pm z$. We may now verify using the table above and reversibility of $M$ that $f$ is in fact a morphism. (For instance, by construction we have $f(a+b) = f(c) = z \in x + y \in f(a) + f(b)$. Similarly, because $x \in z - y$ we have $f(c-b) \subseteq f(c) + f(-b)$.) This provides an inverse to the map $\cMsc(\E_{\cMsc}, M) \to E(M)$ defined above, showing that it is bijective.
\end{proof}

We have seen that the categories $\HMag$ and $\uHMag$ have free objects, and it will be verified in Theorem~\ref{thm:free reversible} that the same is true for $\Msc$ and $\cMsc$. In any of these categories, let $F_2 = F(\{a,b\})$ denote the free object on two elements $a$ and $b$. The inclusion $\{a,b\} \subseteq \E_\C$ then induces a morphism
\[
\iota_\C \colon F_2 \hookrightarrow \E_\C,
\]
which is injective for each choice of $\C$ above. 

This leads to the following characterization of strict morphisms in terms of lifting property, which we find reminiscent of various valuative criteria in algebraic geometry. 

\begin{theorem}\label{thm:strict morphisms}
Let $\C$ denote any of the categories $\HMag$, $\uHMag$, $\Msc$, or $\cMsc$, and retain the notation of $F_2$, $\E := \E_\C$, and $\iota := \iota_\C$ as above. Suppose that $f \colon M \to N$ is a morphism in $\C$. Then the following are equivalent:
\begin{enumerate}[label=\textnormal{(\alph*)}]
\item $f$ is strict;
\item Given any morphisms $\alpha \in \C(F_2, M)$ and $\beta \in \C(\E_\C,N)$ satisfying $f \circ \alpha = \beta \circ \iota_\C$, there exists $g \in \C(\E_\C, M)$ such that
\[
\begin{tikzcd}
	{F_2} && M \\
	\\
	\E && N
	\arrow["f", from=1-3, to=3-3]
	\arrow["\alpha", from=1-1, to=1-3]
	\arrow["\iota_\C"', hook, from=1-1, to=3-1]
	\arrow["\beta"', swap, from=3-1, to=3-3]
	\arrow["{\exists g}", dashed, from=3-1, to=1-3]
\end{tikzcd}
\]
forms a commuting diagram;
\item Every morphism from $\iota_\C$ to $f$ in the arrow category~\cite[II.4]{MacLane} $\cat{Arr}(\C) = \C^{\mathbf{2}}$ factors through the morphism 
\[
\begin{tikzcd}
M \ar[r, equal] \ar[d, equal] & M \ar[d, "f"] \\
M  \ar[r, "f"] & N
\end{tikzcd}
\]
from $\id_M$ to $f$.
\end{enumerate}
\end{theorem} 

\begin{proof}
(a)$\implies$(b): First assume that $f$ is strict and that $\alpha$ and $\beta$ are as in~(b). We must find $g : \E \to M$ that makes the diagram commute. Write $x:=\alpha(a)$ and $y:=\alpha(b)$. By the equality $f \circ \alpha = \beta \circ \iota$, we have $f(x)=\beta(a)$ and $f(y)=\beta(b)$. By strictness of $f$, 
\[
\beta(c) \in \beta(a) \star \beta(b) = f(x) \star f(y)=f(x \star y).
\]
Thus there exists $z \in x \star y$ such that $\beta(c)=f(z)$. So we may define $g \colon \E \to M$ by  setting $g(a) = x$, $g(b) = y$, and $g(c) = z$. Then $g$ makes the diagram of~(b) commute, since:
\begin{gather*}
g \circ \iota(a)=x=\alpha(a), \quad g\circ\iota(b)=y=\alpha(b), \\
f\circ g(a)=f(x)=\beta(a), \  f\circ g(b)=f(y)=\beta(b), \ f\circ g(c)=f(z)=\beta(c).
\end{gather*}

(b)$\implies$(a) Next, assume that $f$ satisfies condition~(b). We wish to verify that $f(x)\star  f(y)\subset f(x \star y)$ for all $x, y\in M$. To this end, fix $z \in f(x) \star f(y)$. Let $\alpha \colon F_2 \to M$ be the morphism defined by $\alpha(a)=x$ and $\alpha(b)=y$. Since $z\in f(x) \star f(y)$, there exists a morphism $\beta \colon \E \to N$ given by $a \mapsto f(x)$, $b \mapsto f(y)$, and $c \mapsto z$. Then we have $f\circ \alpha = \beta \circ \iota$. By hypothesis we obtain $g \colon \E \to M$ such that $\alpha = g\circ \iota$ and $\beta =f\circ g$. These imply that $x=g(a)$, $y=g(b)$, and $z =\beta(c) =f \circ g(c)$. But then
\[
z = f(g(c)) \in f(g(a) \star g(b)) = f(x \star y)
\]
as desired.

The equivalence (b)$\iff$(c) is a formal reinterpretation.
\end{proof}

Incidentally, the objects $\E_\C$ also serve to characterize short morphisms in terms of the following lifting property.

\begin{proposition}\label{prop:short lifting}
Suppose that $\C$ is any of the categories $\uHMag$, $\Msc$, or $\cMsc$ with corresponding object $\E = \E_\C$ as in Lemma~\ref{lem:representing object}. A morphism $p \colon M \to N$ in $\C$ is short if and only if, for every morphism $f \in \C(\E,N)$ there exists $g \in \C(\E,M)$ such that $f = p \circ g$, i.e.~the map $p_* \colon \C(\E, M) \to \C(\E, N)$ is surjective.
\end{proposition}

\begin{proof}
The map $p_*$ corresponds to the following map on the representable functors:
\begin{align*}
\{(u,v,w) \in M^3 \mid w \in u \star v\} \cong \C(\E, M) &\to \C(\E, N) \cong \{(x,y,z) \in N^3 \mid z \in x \star y\}, \\
(u,v,w) &\mapsto (p(u), p(v), p(w)).
\end{align*}
If this function is surjective, note that $p$ is surjective since, for any $n \in N$, we may fix $(x,y,z) = (e,n,n)$ which means that there exists $w \in M$ such that $p(w) = n$.

Then $p$ is short if and only if it is surjective and, for every $x,y \in N$ we have $x \star y \subseteq p(p^{-1}(x) \star p^{-1}(y))$. The latter condition occurs if and only for all $z \in x \star y$ there exist $u \in p^{-1}(x)$, $w \in p^{-1}(y)$, $w \in p^{-1}(z)$ such that $w \in u \star v$. So $p$ is short if and only if it is surjective and the function above is surjective. But surjectivity of $p$ is implied by the surjectivity of the function above, so we deduce that $p$ is short if and only if the lifting property is satisfied.
\end{proof} 

Note that the above characterization is readily extended to $\C = \HMag$ if we include the assumption that $p$ is surjective.

\subsection{Closed monoidal structures}
\label{sub:closed monoidal}

Just as $\HMag$ has its limits and colimits lifted along the forgetful functor $U$ to $\Set$, so we now describe a closed monoidal structure that is similarly lifted via $U$.
For hypermagmas $G$ and $H$, one can naively endow the set of all morphisms from $G$ to $H$ with a hyperoperation in the following way: for $f,g \in \HMag(G,H)$, define
\begin{equation}\label{eq:hyperoperation}
f \star g := \{h \in \HMag(G,H) \mid h(x) \in f(x) \star_H g(x) \mbox{ for all } x \in G\}.
\end{equation}
Under the usual identification of the set of all functions $f \colon G \to H$ with a Cartesian power as
\[
\Set(G,H) = H^G,
\]
the hyperoperation of the product $H^G = \prod_{x \in G} H$ can be viewed as a hyperoperation 
on the set of functions $\Set(G,H)$. Then~\eqref{eq:hyperoperation} is simply the restriction of this hyperoperation to the subset $\HMag(G,H) \subseteq \Set(G,H)$.

Let $X$ and $Y$ be hypermagmas. We will define a hypermagma $X \boxdot Y$ whose underlying set is $X \times Y$. Its elements are denoted $x \boxdot y := (x,y) \in X \boxdot Y$. Given subsets $S \subseteq X$ and $T \subseteq Y$, we will use the shorthand $S \boxdot T = \{s \boxdot t \mid (s,t) \in S \times T\} \subseteq X \boxdot Y$.
With this convention, the hyperoperation of $X \boxdot Y$ is given by
\[
(x \boxdot y) \star (x' \boxdot y') = 
\begin{cases}
x \boxdot (y \star y'), & x = x' \mbox{ and } y \neq y', \\
(x \star x') \boxdot y, & x \neq x' \mbox{ and } y = y', \\
(x \star x) \boxdot y \, \cup \, x \boxdot (y \star y), & (x,y) = (x',y'), \\
\varnothing, & \mbox{otherwise}.
\end{cases}
\] 
Note that this hyperoperation is defined with the smallest possible products of ordered pairs satisfying the properties
\begin{align}
x \boxdot (y_1 \star y_2) &\subseteq x \boxdot y_1 \, \star \, x \boxdot y_2, \label{eq:right dist}\\
(x_1 \star x_2) \boxdot y &\subseteq x_1 \boxdot y \, \star \, x_2 \boxdot y. \label{eq:left dist}
\end{align}

A routine case-by-case argument verifies that the assignment $(X,Y) \mapsto X \boxdot Y$ forms a bifunctor
\[
- \boxdot - \colon \HMag \times \HMag \to \HMag.
\]
This bifunctor can be understood as a kind of tensor product in the sense of~\cite{Jagiello:tensor, Pop:concrete}, as we will show in Proposition~\ref{prop:bimorphisms} below: it represents the appropriate generalization of ``bilinear maps'' in this context. We recall the following definition in the general context of concrete categories as we will use it in several different categories. 

\begin{definition}
Let $(\catC, U)$ be a concrete category, i.e.~a category with a faithful functor $U \colon \catC \to \Set$.  For objects $X,Y,Z \in \catC$, a function $B \colon U(X) \times U(Y) \to U(Z)$ is a \emph{bimorphism} if
\begin{itemize}
\item for all $x \in X$, the map $B(x,-) \colon U(Y) \to U(Z)$ satisfies $B(x,-) = U(f_x)$ for some $f_x \in \catC(Y,Z)$, and
\item for all $y \in Y$, the map $B(-,y) \colon U(X) \to U(Z)$ satisfies $B(-,y) = U(g_y)$ for some $g_y \in \catC(X,Z)$.
\end{itemize}
The set of all bimorphisms $B \colon X \times Y \to Z$ will be denoted by $\Bim(X,Y; Z) = \Bim_\catC(X,Y;Z)$. Then we obtain a multifunctor 
\[
\Bim(-,- \, ; -) \colon \catC\op \times \catC\op \times \catC \to \Set
\]
in a straightforward way.
\end{definition}

In all instances below, $U$ will be the obvious forgetful functor and it will be suppressed throughout.

\begin{proposition}\label{prop:bimorphisms}
For hypermagmas $X$ and $Y$, the hypermagma $X \boxdot Y$ represents the functor of bimorphisms $\Bim(X,Y; -)$. That is, we have bijections
\[
\HMag(X \boxdot Y, Z) \cong \Bim(X,Y; Z)
\] 
that are natural in all $X,Y,Z \in \HMag$.
\end{proposition}

\begin{proof}
It follows from~\eqref{eq:right dist} and~\eqref{eq:left dist} that the identity function gives a bimorphism $u \colon X \times Y \to X \boxdot Y$. The pair $(X \boxdot Y, u)$ is easily seen to be initial among all pairs $(Z,B)$ where $Z$ is a hypermagma with a bimorphism $B \colon X \times Y \to Z$; more precisely, it is initial in the category of elements $\operatorname{el}(\Bim(X,Y; -))$. It follows~\cite[Proposition~2.4.8]{Riehl:context} that $u \in \Bim(X,Y; X \boxdot Y)$ is the Yoneda element corresponding to a natural isomorphism $\HMag(X \boxdot Y, -) \cong \Bim(X,Y; -)$.
\end{proof}

We let $1_\varnothing$ denote the terminal set $1$ equipped with the empty hyperoperation: if we denote its unique element by~$1$ (a slight abuse of notation), this hyperoperation is given by $1 \star 1 = \varnothing$. Note that this is the free hypermagma on one element.

\begin{theorem}\label{thm:HMag closed}
The symmetric monoidal category $(\HMag, \boxdot, 1_\varnothing)$ is closed, with internal hom given by $[M,N] := \HMag(M,N)$ under the hyperoperation~\eqref{eq:hyperoperation}.
\end{theorem}

\begin{proof}
The fact that $\boxdot$ and $1_\varnothing$ induce a symmetric monoidal structure largely follows from the fact that $(\Set, \times, 1)$ is symmetric monoidal. One subtle point is the fact that $1_\varnothing$ forms a monoidal unit. If $M$ is a hypermagma, we wish to verify that the bijection $1_\varnothing \boxdot M \to M$ given by $1 \boxdot m \mapsto m$ is an isomorphism, which is to say that $1 \boxdot m \star 1 \boxdot m' = 1 \boxdot (m \star m')$ for all $m,m' \in M$. This follows immediately by construction as long as $m \neq m'$. Fortunately in the case where $m = m'$, we also have
\begin{align*}
1 \boxdot m \, \star \, 1 \boxdot m &= (1 \star 1) \boxdot m \, \cup \, 1 \boxdot (m \star m) \\
&= \varnothing \boxdot m \, \cup \, 1 \boxdot (m \star m) \\
&= 1 \boxdot (m \star m).
\end{align*}

It remains to show that this symmetric monoidal category is closed with the structure described in the statement. Fix hypermagmas $X,Y,Z \in \HMag$, and consider the natural bijection from the Cartesian closed structure of $\Set$:
\[
\begin{tikzcd}[row sep=small, column sep=large]
\Set(X \times Y, Z) \arrow[r, rightarrow, "\sim"] & \Set(X, Z^Y) \\
\HMag(X \boxdot Y, Z) \arrow[r, rightarrow, dashed, "\sim"]  \arrow[u, phantom, sloped, "\subseteq"] 
  & \HMag(X, \HMag(Y,Z)) \arrow[u, phantom, sloped, "\subseteq"] 
\end{tikzcd}
\]
Under this bijection, an element $\phi \in \HMag(X \boxdot Y, Z)$ corresponds to the function 
\begin{align*}
\widehat{\phi} \colon X &\to Z^Y, \\
x &\mapsto \phi(x \boxdot -). 
\end{align*}
Each map $\phi(x \boxdot -) \colon Y \to Z$ is a morphism of hypermagmas thanks to~\eqref{eq:right dist}. Thus after corestriction we may view the function corresponding to $\phi$ as a mapping $\widehat{\phi} \colon X \to \HMag(Y,Z)$. It is a direct consequence of~\eqref{eq:left dist} that this $\widehat{\phi}$ is a morphism of hypermagmas, so that in fact $\widehat{\phi} \in \HMag(X, \HMag(Y,Z))$.

Thus (co)restriction of the Cartesian closed structure on $\Set$ induces the dashed arrow in the diagram above. To see that it is bijective, fix $\psi \in \HMag(X, \HMag(Y,Z))$ whose action we denote by $x \mapsto \psi_x \in \HMag(Y,Z)$. The reader can verify that the function $X \times Y \to Z$ given by $(x,y) \mapsto \psi_x(y)$ is a bimorphism. By Proposition~\ref{prop:bimorphisms} this corresponds to a morphism $\psi_0 \in \HMag(X \boxdot Y, Z)$, which is readily seen to satisfy $\widehat{\psi_0} = \psi$. Similarly, for $\phi \in \HMag(X \boxdot Y, Z)$ one can verify that $( \widehat{\phi} )_0 = \phi$. So the dashed arrow above is bijective, and we obtain a natural isomorphism $\HMag(X \boxdot Y, Z) \cong \HMag(X, \HMag(Y,Z))$ as desired.
\end{proof}

\separate

Next we describe how the closed monoidal structure of $\HMag$ descends to a closed monoidal structure on $\uHMag$. Similar to the non-unital case, this structure can be viewed as an enrichment of the closed monoidal structure $(\Set_\bullet, \wedge, \two)$ of pointed sets. We briefly recall that the monoidal product here is the \emph{wedge product} (or \emph{smash} product) $(X, x_0) \wedge (Y, y_0)$ which is the quotient of the product $X \times Y$ by the equivalence relation $(x,y_0) \sim (x_0, y_0) \sim (x_0, y)$ for all $x \in X$ and $y \in Y$, and that the internal hom is the set $\Set_\bullet((X,x_0),(Y,y_0))$ with basepoint given by the constant function $y_0$.
This monoidal structure has unit $(\two, 0)$, where $\two = \{0,1\}$ is the two-element set.

As in the non-unital case, the internal hom is easier to describe. Let $M$ and $N$ be unital hypermagmas. We may (co)restrict the operation~\eqref{eq:hyperoperation} to the subset $\uHMag(M,N) \subseteq \HMag(M,N)$ as follows: for $f,g \in \uHMag(M,N)$, define 
\[
f \star g = \{h \in \uHMag(M,N) \mid h(x) \in f(x) \star g(x) \mbox{ for all } x \in M\}.
\]
This is a hyperoperation on the set of unit-preserving morphisms. If $e$ denotes the identity element of $N$, it is then clear that the constant function  $e \colon M \to N$ is an identity for $\uHMag(M,N)$.

To describe the symmetric monoidal structure, continue to let $M$ and $N$ denote unital hypermagmas. Within the hypermagma $M \boxdot N$ we fix the strict subhypermagma
\[
E = M \boxdot e_N \cup e_M \boxdot N \subseteq M \boxdot N.
\]
Then we define $M \boxwedge N := (M \boxdot N)_E$ to be the unitization relative to this subset. By construction, this is a unital hypermagma with unit $e = \pi_E(E) = \pi_E(e_M \boxdot e_N) \in M \boxwedge N$. The composite surjection
\[
\begin{tikzcd}
- \boxwedge - \colon M \times N \ar[r, "-\boxdot-"] & M \boxdot N \ar[r, "\pi_E"] & M \boxwedge N
\end{tikzcd}
\]
is given by a bimorphism followed by a hypermagma morphism. Thus it is a bimorphism of hypermagmas $M \times N \to M \boxwedge N$. Furthermore, because
\begin{align*}
x \boxwedge e_N = e = e_M \boxwedge y
\end{align*}
for all $x \in M$ and $y \in N$, it follows that $- \boxwedge -$ is in fact a bimorphism of \emph{unital} hypermagmas.

Next we verify that the underlying set of this object coincides with the smash product of the pointed sets.

\begin{lemma}
The canonical map $\pi_E \colon M \boxdot N \to M \boxwedge N$ satisfies 
\[
\pi_E^{-1}(x \boxwedge y) = 
\begin{cases}
E, & x \boxwedge y = e, \\
\{x \boxdot y\}, & x \boxwedge y \neq e.
\end{cases}
\]
In particular, the underlying pointed set of $M \boxwedge N$ is given by the smash product $(M, e_M) \wedge (N, e_N)$.
\end{lemma}

\begin{proof}
Recall from the proof of Lemma~\ref{lem:unitization} that the unitization with respect to $E$ is constructed as the quotient by the finest equivalence relation $\sim$ that contains $E \times E$ and satisfies a certain condition. The claim therefore amounts to showing that the particular equivalence relation 
\[
{\sim} = (E \times E) \sqcup \Delta_{(M \boxdot N) \setminus E}
\]
satisfies the extra condition. The fact that $E$ is a strict subhypermagma simplifies this verification to the following: 
we may assume that $u \in E$ and $x \in (M \boxdot N) \setminus E$, and we wish to prove that
\[
y \in x \star u \cup u \star x \implies y = x.
\]
We must have $x = x_1 \boxdot x_2$ where $x_1 \neq e_M$ and $x_2 \neq e_N$. Suppose first that $u$ has the form $u = m \boxdot e_N$. Then if $x \star u$ is nonempty we have
\begin{align*}
x \star u &=  (x_1 \boxdot x_2) \star (m \boxdot e_N) \neq \varnothing \\
 &\implies x_1 = m \quad (\mbox{since } x_2 \neq e_N) \\
 &\implies  (x_1 \boxdot x_2) \star (m \boxdot e_N)= x_1 \boxdot (x_2 \star e_N) = x_1 \boxdot x_2.
\end{align*}
Thus $y \in x \star u$ implies $y = x$, and symmetrically $y \in u \star x$ implies $y = x$. A similar argument holds if $u$ is of the form $u = e_M \boxdot n$, which completes the proof as $E = M \boxwedge e_N \cup e_M \boxwedge N$.
\end{proof}

In the following, we view the set $\two = \{0,1\}$ as a unital hypermagma under the  hyperoperation $\star$ for which $0$ is the identity and $1 \star 1 = \varnothing$. This is also the free unital hypermagma on the singleton $\{1\}$. 

\begin{theorem}\label{thm:unital tensor}
For $M,N \in \uHMag$, the unital hypermagma $M \boxwedge N$ represents the functor of bimorphisms
\[
\uHMag(M \boxwedge N, -) \cong \Bim(M, N; -) \colon \uHMag \to \Set.
\]
The symmetric monoidal category $(\uHMag, \boxwedge, \two)$ is closed, with internal hom given by $[M,N] := \uHMag(M,N)$ under the hyperoperation inherited from~\eqref{eq:hyperoperation}.
\end{theorem}

\begin{proof}
The universal property of $M \boxwedge N$ is proved following the same argument as in Proposition~\ref{prop:bimorphisms}, by showing that the canonical map $M \times N \to M \boxwedge N$ is a bimorphism in $\uHMag$. This makes it evident that we obtain a bifunctor 
\[
- \boxwedge - \colon \uHMag \times \uHMag \to \uHMag.
\]
The rest of the claim can be proved in the same manner as Theorem~\ref{thm:HMag closed}, with only two significant adjustments described below.

First, note that the monoidal unit is now the unital hypermagma $\two$ described above. For a unital hypermagma $M$, one can check that the isomorphism $\two \boxwedge M \cong M$ of underlying pointed sets is in fact an isomorphism in $\uHMag$. The computation is similar, noting that the strict subhypermagma $\{1\} \subseteq \two$ is isomorphic to $1_\varnothing$.

Second, the argument regarding the closed monoidal structure is proved using the forgetful functor $U \colon \uHMag \to \Set_\bullet$ to pointed sets rather than sets.
Let $X$, $Y$, and $Z$ denote unital hypermamgas, which we view as pointed sets whose basepoint is the identity (thereby suppressing the forgetful functor $U$ in notation below). Then the natural isomorphism
\[
\uHMag(X \boxwedge Y, Z) \cong \uHMag(X, \uHMag(Y,Z))
\]
can be deduced by (co)restriction from the closed monoidal structure on $\Set_\bullet$ via the diagram
\[
\begin{tikzcd}[row sep=small, column sep=large]
\Set_\bullet(X \wedge Y, Z) \arrow[r, rightarrow, "\sim"] & \Set_\bullet(X, \Set_\bullet(Y,Z)) \\
\uHMag(X \boxwedge Y, Z) \arrow[r, rightarrow, dashed, "\sim"]  \arrow[u, phantom, sloped, "\subseteq"] 
  & \uHMag(X, \uHMag(Y,Z)) \arrow[u, phantom, sloped, "\subseteq"] 
\end{tikzcd}
\]
following an argument that is analogous to the one given in Theorem~\ref{thm:HMag closed}.
\end{proof}

\separate

Finally, we wish to remark that these closed monoidal structures restrict well to the subcategories of commutative objects. Indeed, let $\cHMag$ and $\cuHMag$ respectively denote the full subcategories of $\HMag$ and $\uHMag$ consisting of the commutative (unital) hypermagmas. 
It is clear from the construction of $M \boxdot N$ that if $M$ and $N$ are commutative, then the same is true for $M \boxdot N$, so that $(\cHMag, \boxdot, 1_\varnothing)$ is a monidal subcategory of $\HMag$. Furthermore, the definition of the operation~\eqref{eq:hyperoperation} is such that if $N$ is commutative, then so is $\HMag(M,N)$. Similar remarks hold for $\cuHMag$ if $M$ and $N$ are unital hypermagmas. Thus we immediately arrive at the following.

\begin{corollary}
The full subcategories $\cHMag$ of $\HMag$ and $\cuHMag$ of $\uHMag$ are both exponential ideals and closed under monoidal products in the respective closed monoidal structures. Consequently, both $(\cHMag, \boxdot, 1_\varnothing)$ and $(\cuHMag, \boxwedge, \two)$ are closed monoidal categories.
\end{corollary}

By contrast, the formulas defining the hyperoperations of $M \boxdot N$ and $M \boxwedge N$ is not generally associative, so that the category $\HMon$ of hypermonoids will not form a monoidal subcategory of $\uHMag$. Furthermore, Theorem~\ref{thm:empty sum} will provide an explicit example of hypergroups $M$ and $N$ such that the unital hypermagma $\uHMag(M,N) = \HGrp(M,N)$ is not associative.

\section{Categories of hypergroups and mosaics}\label{sec:mosaics}

In this final section, we consider categories of mosaics and hypergroups.
One of the lessons learned will be that while (canonical) hypergroups are attractive objects, their categories are unfortunately less well-behaved than those of other structures. The culprit in this situation is apparently the associative axiom, since the categories of (commutative) mosiacs are quite nicely behaved. For this reason we argue that the categories $\Msc$ and $\cMsc$ can be taken as ``convenient'' repacements for $\HGrp$ and $\Can$, respectively.

\subsection{A convenient category of (canonical) hypergroups}
\label{sub:convenient}

In this subsection we describe how many of the good properties of $\uHMag$ are also enjoyed by the full subcategories $\Msc$ and $\cMsc$. To begin, we note that the categories are complete and cocomplete in the following way.

\begin{theorem}\label{thm:reversible forgetful}
The subcategories $\Msc$ and $\cMsc$ are closed under limits and colimits in $\uHMag$. Thus they are complete and cocomplete.
\end{theorem}

\begin{proof}
It suffices to prove that $\Msc$ is closed under (co)products and (co)equalizers. Before doing so, we make the following general observation.

Recall from Lemma~\ref{lem:inversion} that the inversion of a mosaic is an anti-isomorphism. Let $D \colon J \to \Msc$ be a diagram over a small index category $J$, and let $D' \colon J \to \Msc$ be the diagram such that each $D'(j) = D(j)\op$ and which assigns the ``opposite'' morphism to each morphism of $J$. Then the inversions $i_j \colon D(j) \to D'(j)$ form components of a natural isomorphism $i \colon D \to D'$ of functors $J \to \Msc \hookrightarrow \uHMag$. It follows that we obtain an induced anti-isomorphism of unital hypermagmas on both the limit $L = \lim_J D$ and colimit $C = \colim_J D$ of the diagram as computed in $\uHMag$. 

Using this general construction of inverses, it is easy to verify that any small product or coproduct of (commutative) mosaics in $\uHMag$ is again a (commutative) mosaic, and that the equalizer of a pair of morphisms between (commutative) mosaics is again such. 
Finally, consider a coequalizer in $\uHMag$ of the form
\[
\begin{tikzcd}
M \ar[r, shift left=1, "f"] \ar[r, shift right =1, swap, "g"] & N \ar[r, twoheadrightarrow, "\pi"] & L
\end{tikzcd}
\]
where $M$ and $N$ are reversible. By the discussion above we obtain an inversion $i \colon L \to L$ satisfying $i(\pi(x)) = \pi(x^{-1})$.
Assume that $x,y,z \in L$ satisfy $x \in y \star z$. Recall from the proof of Theorem~\ref{thm:unital cocomplete} that $\pi$ is short, so that 
\[
x \in y \star z=\pi(\pi^{-1}(y)\star \pi^{-1}(z))
\]
Thus there are $x', y', z' \in N$ such that $x=\pi(x')$, $y=\pi (y')$, $z=\pi(z')$ and $x'\in y'\star z'$. By reversibility of $N$ we have $y' \in x'\star (z')^{-1}$ and $z' \in (y')^{-1} \star x'$. Applying $\pi$, we obtain 
\begin{align*}
y &= \pi(y') \in \pi(x'\star (z')^{-1}) \subseteq \pi(x') \star \pi((z')^{-1}) = x \star i(z),\\
z &= \pi(z') \in \pi((y')^{-1} \star x')\subseteq \pi((y')^{-1})\star \pi(x') = i(y)\star x.
\end{align*}
Therefore $L$ is reversible. (Finally, if $N$ is commutative then so is $L$ thanks to Lemma~\ref{lem:short image}(2).)
\end{proof}

We see immediately that the characterizations of various epimorphisms and monomorphisms as well as the property of regularity easily pass from unital hypermagmas to (commutative) mosaics.

\begin{corollary}\label{cor:mosaic morphisms}
In the categories $\Msc$ and $\cMsc$,
\begin{itemize}
\item The monomorphisms and epimorphisms are the injective and surjective morphisms, respectively;
\item The regular monomorphisms and regular epimorphisms are the coshort and short morphisms, respectively;
\item The normal monomorphisms and normal epimorphisms correspond to the strict submosaics and unitizations, respectively.
\end{itemize}
Furthermore, these categories are both regular.
\end{corollary}

\begin{proof}
The properties of being an ordinary, regular, or normal monomorphism (resp., epimorphism) are all characterized in terms of a limit (resp., colimit). Then thanks to Theorem~\ref{thm:reversible forgetful}, a morphism in $\Msc$ or $\cMsc$ satisfies any one of these properties in $\Msc$ or $\cMsc$ if and only if it satisfies the corresponding property in $\uHMag$. Thus the characterizations of all morphisms above, with the exception of normal monomorphisms, follow directly from Proposition~\ref{prop:mono and epi}, Theorem~\ref{thm:regular epi and mono}, and Theorem~\ref{thm:normal}. The fact that normal monomorphisms correspond to strict submosaics follows from Theorem~\ref{thm:normal}, and Remark~\ref{rem:absorptive}(5,6).

Finally, because regular categories are wholly characterized in terms of certain finite limits and colimits, the subcategories $\Msc$ and $\cMsc$ inherit the property of reguarity from $\uHMag$ (Corollary~\ref{cor:regular}). 
\end{proof}

Next we show that free objects exist in $\Msc$, and that they are even commutative.

\begin{theorem}\label{thm:free reversible}
The forgetful functor $U \colon \Msc \to \Set$ has a left adjoint.
\end{theorem}

\begin{proof}
Let $X$ be a set. Let $F(X) = (X \times \{0,1\}) \sqcup \{0\}$, equipped with the involution $- \colon F(X) \to F(X)$ that fixes zero and interchanges $(x,0)$ with $(x,0)$ for all $x \in X$. Thus if we identify $X$ with $X \times \{0\}$, we may view $F(X) = X \sqcup -X \sqcup \{0\}$.

Define a hyperaddition on $F(X) \setminus \{0\}$ by
\[
a + b = 
\begin{cases}
0 & \mbox{if } a = -b, \\
\varnothing &\mbox{otherwise}
\end{cases}
\]
and extend it to $F(X)$ by setting $0$ to be an additive identity. Then it is straightforward to verify that $F(X)$ is an object of $\cMsc$, that this construction gives a functor $F \colon \Set \to \cMsc \subseteq \Msc$, and that we obtain natural bijections
\[
\Set(X, U(M)) \cong \Msc(F(X), M)
\]
for all $X \in \Set$ and $M \in \Msc$.
\end{proof}

\separate

The additive structure of the Krasner hyperfield (Example~\ref{ex:Krasner}) plays the following special role in the category of mosaics. 

\begin{proposition}\label{prop:strict classifier}
The functor $\Sub\str \colon \Msc\op \to \Set$ of strict submosaics is representable by the Krasner hyperfield: 
\[
\Sub\str \cong \Msc(-, \K).
\]
\end{proposition}

\begin{proof}
Let $G$ be a mosaic. Certainly if $\phi \in \Msc(G, \K)$ then $N = \phi^{-1}(0) \subseteq G$ is a strict submosaic. 
Conversely, suppose that $N \subseteq G$ is a strict submosaic. Define a function $\phi = \phi_N \colon G \to \K$ by
\[
\phi(x) =
	\begin{cases}
	0, & x \in N, \\
	1, & x \notin N.
	\end{cases}
\]
We claim that $\phi$ is a morphism of mosaics. Clearly $\phi$ preserves the unit. Given $x,y \in G$, we wish to show that
\[
\phi(x \star y) \subseteq \phi(x) + \phi(y).
\]
If $x,y \in N$ then 
\[
\phi(x \star y) \subseteq \phi(N) = 0 = 0 + 0 =  \phi(x) + \phi(y), 
\]
and if $x,y \notin N$ then 
\[
\phi(x \star y) \subseteq \{0,1\} = 1 + 1 = \phi(x) + \phi(y).
\]
So assume that $x \in N$ and $y \notin N$. It suffices to show show that $x \star y \subseteq G \setminus N$, for then it will follow that 
\[
\phi(x \star y) \subseteq \phi(G \setminus N) = 1 = 0 + 1 = \phi(x) + \phi(y).
\]
So let $z \in x \star y$, and assume toward a contradiction that $z \in N$. Then reversibility implies that $y \in x^{-1} \star z \in N$, which is a contradiction. A symmetric argument applies if $x \notin N$ and $y \in N$.

Thus we have a bijection
\begin{align*}
\Sub\str(G) &\cong \Msc(G,\K), \\
N &\mapsto \phi_N,
\end{align*}
which is evidently natural in $G$. Thus $\K$ (with its zero submosaic) represents $\Sub\str$. 
\end{proof}


\separate

Next we describe a closed monoidal structure on the category of commutative mosaics, reminiscent to that of the category $(\Ab, \otimes, \Z)$ of abelian groups. First note that the internal hom of $\uHMag$ naturally induces an internal hom on $\cMsc$, the hyperoperation on $f,g \in \cMsc(M,N)$ by
\begin{equation}\label{eq:additive hom}
f + g = \{h \in \cMsc(M,N) \mid h(x) \in f(x) + g(x) \mbox{ for all } x \in M\}.
\end{equation}
While we know this has identity given by the constant zero morphism, one can readily verify that it is also reversible: the unique inverse of $f$ is the morphism $-f \colon M \to N$ given by $(-f)(x) = -f(x)$ for $x \in M$.

Now we turn to the construction of the corresponding monoidal structure.
For any object $M \in \cMsc$, the negation map $-1 \colon M \to M$ given by $m \mapsto -m$ is a morphism (thanks to commutativity and uniqueness of inverses). Thus for two objects $M,N \in \cMsc$ we may form the endomorphism $(-1) \boxwedge (-1)$ of $M \boxwedge N$. We let $M \boxtimes N$ denote the coequalizer of the endomorphisms $i_{++} := \id_M \boxwedge \id_N = \id_{M \boxwedge N}$ and $i_{--} = (-1) \boxwedge (-1)$ in $\uHMag$: 
\begin{equation}
\begin{tikzcd}
M \boxwedge N \ar[r, shift left=.5ex, "i_{++}"] \ar[r, shift right=.5ex, swap, "i_{--}"] & M \boxwedge N \ar[r, twoheadrightarrow] & M \boxtimes N
\end{tikzcd}
\end{equation}
Thus the underlying set of $M \boxtimes N$ is the quotient of $M \boxwedge N$ by the equivalence relation $x \boxwedge y \sim (-x) \boxwedge (-y)$ (where $x \in M$, $y \in N$), which is equivalently described by $(-x) \boxwedge y \sim x \boxwedge (-y)$. Thus we may define an involution on $M \boxtimes N$ by setting
\[
- m \boxtimes n := (-m) \boxtimes n = m \boxtimes (-n).
\]
This provides an inverse for each element:
\[
0 = 0 \boxtimes n \in (-m + m) \boxtimes n = - m \boxtimes n + m \boxtimes n.
\]
By construction, there is a bijection 
\begin{equation}\label{eq:tensor set}
M \boxtimes N \cong (M \wedge N)/(\Z/2\Z),
\end{equation}
where the group $\Z/2\Z$ acts on the pointed set $M \wedge N$ by $m \wedge n \mapsto (-m) \wedge (-n)$. In particular, we immediately see the following nondegeneracy condition of this ``tensor product'' that behaves more like the case of vector spaces than the case of abelian groups:
\[
x \in M \setminus \{0\}, y \in N \setminus \{0\} \implies x \boxwedge y \neq 0.
\]

\begin{lemma}
If $M$ and $N$ are commutative mosaics, then $M \boxtimes N$ is also a commutative mosaic.
\end{lemma}

\begin{proof}
By construction $M \boxtimes N$ is a commutative unital hypermagma with inverses as described above. 
Thus it only remains to check reversibility. Before doing so, consider that the surjective morphisms of hypermagmas
\[
M \boxdot N \twoheadrightarrow M \boxwedge N \twoheadrightarrow M \boxtimes N
\]
are both short, so that the composite surjection $\pi \colon M \boxdot N \twoheadrightarrow M \boxtimes N$ is also short.

To verify reversibility, suppose that $w_1, w_2, w_3 \in M \boxtimes N$ are elements such that 
\[
w_1 \in w_2 + w_3 = \pi(\pi^{-1}(w_2) + \pi^{-1}(w_3)).
\]
This means that there exist $w'_i = x_i \boxdot y_i \in M \boxdot N$ such that each $\pi(w'_i) = w_i$ and 
\[
x_1 \boxdot y_1 \in (x_2 \boxdot y_2) + (x_3 \boxdot y_3)
\]
Recalling the structure of $M \boxdot N$ described above, we must have either $x_2 = x_3$ or $y_2 = y_3$ for the sum on the right-hand-side to be nonempty. We may separate the argument into one of a few cases.

First suppose that $x_2 = x_3 =: x$ but $y_2 \neq y_3$. Then by reversibility of $N$ we have
\begin{align*}
x_1 \boxdot y_1 \in x \boxdot (y_2 + y_3) &\implies x_1 = x \mbox{ and } y_1 \in y_2 + y_3 \\
&\implies y_3 \in y_1 - y_2 \\
&\implies x \boxdot y_3 \in x \boxdot y_1 + x \boxdot (-y_2).
\end{align*}
Applying $\pi$, we then deduce that $w_3 \in w_1 - w_2$. A similar argument yields the same conclusion in the case where $x_2 \neq x_3$ and $y_2 = y_3$.

Finally, suppose that $x_2 = x_3 =:x $ and $y_2 = y_3 =: y$. Then we have
\[
x_1 \boxdot y_1 \in (x \boxdot y) + (x \boxdot y) = (x+x) \boxdot y \, \cup \, x \boxdot (y + y).
\]
Thus we either have $x_1 = x$ and $y_1 \in y + y$, or we have $x_1 \in x + x$ and $y_1 \in y + y$. In the first case, reversibility gives $y \in y_1 - y$, so that 
\[
x \boxdot y \in x \boxdot (y_1 -  y) = x_1 \boxdot y_1 + x \boxdot (-y).
\]
Applying $\pi$ thus implies that $w_3 \in w_2 - w_1$. A similar argument in the second case derives the same conclusion. 

Thus by exhaustion of all cases we conclude that $M \boxtimes N$ is in fact reversible.
\end{proof}

\begin{theorem}\label{thm:mosaic bimorphisms}
For commutative mosaics $M$ and $N$, the object $M \boxtimes N \in \cMsc$ represents the functor of bimorphisms:
\[
\cMsc(M \boxtimes N, -) \cong \Bim_{\cMsc}(M,N; -) \colon \cMsc \to \Set.
\]
\end{theorem}

\begin{proof}
Suppose that $L$ is another object in $\cMsc$. Since $\cMsc$ is a full subcategory of $\uHMag$, a function $M \times N \to L$ is a bimorphism in $\uHMag$ if and only if it is a bimorphism in $\cMsc$.
Theorem~\ref{thm:unital tensor} gives a natural isomorphism 
\begin{equation}\label{eq:bimorphisms}
\uHMag(M \boxwedge N, L) \cong \Bim_{\uHMag}(M,N;L) = \Bim_{\cMsc}(M,N;L).
\end{equation}
But also because any bimorphism $B \colon M \times N \to L$ satisfies 
\[
B(-m,n) = -B(m,n) = B(m,-n),
\] 
it follows that any morphism $f \colon M \boxwedge N \to L$ satisfies 
\[
f((-m) \boxwedge n) = f(m \boxwedge (-n)) = -f(m \boxwedge n). 
\]
Thus $f$ coequalizes the endomorphisms $i_{++}$ and $i_{--}$ of $M \boxwedge N$ described above, so that it factors uniquely via a morphism out of the coequalizer $\overline{f} \colon M \boxtimes N \to L$. Thus $f \mapsto \overline{f}$ provides a natural bijection
\[
\uHMag(M \boxwedge N, L) \cong \cMsc(M \boxtimes N, L),
\]
which combines with~\eqref{eq:bimorphisms} to yield the desired representability.
\end{proof}

Recall from Theorem~\ref{thm:free reversible} that free objects exist in $\Msc$ and that they are commutative, so that they also form free objects in $\cMsc$.
Let $\FF = \{1, 0, -1\}$ denote the free object of $\cMsc$ generated by the single element $1$, so that $1+1 = -1-1 = \varnothing$ and $1 - 1 = 0$.

\begin{theorem}\label{thm:mosaic monoidal}
The symmetric monoidal category $(\cMsc, \boxtimes, \FF)$ is closed, with internal hom given by $[M,N] := \cMsc(M,N)$ under the hyperoperation~\eqref{eq:additive hom}.
\end{theorem}

\begin{proof}
Fix objects $M,N,L \in \cMsc$. We claim that there is a natural bijection between morphisms $M \to \cMsc(N,L)$ and bimorphisms $M \times N \to L$. From this and Theorem~\ref{thm:mosaic bimorphisms} will follow the adjunction
\[
\cMsc(M \boxtimes N, L) \cong \Bim_{\cMsc}(M,N;L) \cong \cMsc(M, \cMsc(N,L)).
\]
Certainly every bimorphism $B \colon M \times N \to L$ determines a morphism $M \to \cMsc(N,L)$ given by $m \mapsto B(M,-)$. Conversely, suppose that 
\begin{align*}
\phi \colon M &\to \cMsc(N,L) \\
m &\mapsto \phi_m
\end{align*}
is a morphism of mosaics. We obtain a function $B \colon M \times N \to L$ by setting $B(m,n) = \phi_m(n)$. This is a bimorphism because we have
\begin{align*}
\phi_{m+m'}(n) &\subseteq \phi_m(n) + \phi_{m'}(n), \\
\phi_m(n+n') &\subseteq \phi_m(n) + \phi_m(n'), \\
\phi_m(0) &= 0 = \phi_0(n)
\end{align*}
for all $m,m' \in M$ and $n,n' \in N$. These assignments are readily verified to be mutually inverse and natural in $M$, $N$, and $L$, so that the adjunction is established.

It remains to show that $\FF$ is a monoidal unit. For any commutative mosaic $M$, we have a bimorphism $\FF \times M \to M$ given by $(\pm 1, m) \mapsto \pm m$ and $(0, m) \mapsto m$. This uniquely determines a morphism $\FF \boxtimes M \to M$, which one can verify is a bijection. One can check that the inverse assignment $M \to \FF \boxtimes M$ given by $m \mapsto 1 \boxtimes m$ is in fact a morphism, with the only subtle observation here is the case
\begin{align*}
1 \boxtimes m + 1 \boxtimes m &= (1 + 1) \boxtimes m \, \cup \, 1 \boxtimes (m + m) \\
&= \varnothing \, \cup \, 1 \boxtimes (m + m) \\
&= 1 \boxtimes (m + m).
\end{align*}
This gives $\FF \boxtimes M \cong M$ naturally in $M$.
\end{proof}

\begin{example}
Suppose that $G$ and $H$ are abelian groups, with (ordinary) tensor product $G \otimes H = G \otimes_\Z H$. Because the canonical map $G \times H \to G \otimes H$ is bilinear, it is also a bimorphism of mosiaics. Thus it induces a morphism of mosaics
\begin{align*}
G \boxtimes H &\to G \otimes H, \\
g \boxtimes h &\mapsto g \otimes h.
\end{align*}
The image of this morphism lies in the weak submosaic consisting of the pure tensors in $G \otimes H$. However, this morphism need not be injective: the underlying set of $G \boxtimes H$ is in bijection with $(G \wedge H)/(\Z/2\Z)$ as in~\eqref{eq:tensor set}, while one can certainly choose abelian groups such that $G \otimes H = 0$.
\end{example}

\begin{remark}
In principle, the construction of the object $M \boxtimes N$ should carry through even if we do not assume that the mosaics are commutative. Extra care is required in the definition of the object in this case: one should instead take the quotient of the unital hypermagma $M \boxwedge N$ by the equivalence relation generated by $x \boxwedge y^{-1} \sim x^{-1} \boxwedge y$ for all $x \boxwedge y \in M \boxwedge N$. However, it seems more difficult to identify the underlying set of this quotient. In particular, there is no reason to expect that $x \in M \setminus \{e\}$ and $y \in N \setminus \{e\}$ would still imply $x \boxtimes y \neq e$. Because of this extra complication, and because our motivation was to mimic the tensor product of abelian groups, we have chosen to focus only on the commutative case here.
\end{remark}

The closed monoidal structure on $\cMsc$ allows for a new view of hyperrings and, more generally, multirings. We recall from~\cite[Section~4]{Viro:hyperfields} that a \emph{multiring} $(R, +, 0, \cdot, 1)$ is a set $R$ equipped with the structures of a canonical hypergroup $(R, +, 0)$ and a monoid $(R, \cdot, 1)$ subject to the condition that $0 \cdot R = 0 = R \cdot 0$ along with the following ``subdistributive'' property: for all $a, b, c \in R$,
\begin{equation}\label{eq:subdistributive}
a(b + c) \subseteq ab + ac \quad \mbox{and} \quad (b+c)a \subseteq ba + ca.
\end{equation}
A multiring $R$ is a \emph{hyperring} if it satisfies the ``strict'' form of distributivity: for all $a,b,c \in R$,
\begin{equation}\label{eq:distributive}
a(b + c) = ab + ac \quad \mbox{and} \quad (b+c)a = ba + ca.
\end{equation}
 A \emph{morphism} of multirings is a function that is both a morphism the canonical hypergroup structure and the multiplicative monoid structure. We let $\MRing$ denote the category of multirings with these morphisms, and we let $\HRing$ denote the full subcategory whose objects are the hyperrings.
 
It is well known that the category of rings can be equivalently viewed as the category of monoid objects~\cite[VII.3]{MacLane} in $(\Ab, \otimes, \Z)$. The following shows that multirings enjoy a similar description. If $\C$ is a monoidal category (whose monoidal structure is understood from context), we let $\Mon(\C)$ denote the corresponding category of monoid objects in $\C$.

\begin{theorem}\label{thm:multirings}
The category of multirings has a fully faithful embedding
 \[
 \cat{MRing} \hookrightarrow \Mon(\cMsc) 
 \]
whose image is the full subcategory of objects with underlying additive mosaic being associative (i.e., a canonical hypergroup). 
\end{theorem}
 
\begin{proof}
Given a multiring $(R, +, 0, \ast_R , 1)$, we define a monoid object $(R, m, \eta)$ of $(\cMsc, \boxtimes, \FF)$ as follows. 
Because the monoidal unit $\FF = \{0, 1, -1\}$ of $\cMsc$ is freely generated by $1$, there is a unique morphism of mosaics $\eta_R \colon \FF \to R$ determined by $1 \mapsto 1$. 
The zero and subdistributive~\eqref{eq:subdistributive} properties of multiplication in $R$ imply that it is a bimorphism in $\cMsc$. By the universal property of Theorem~\ref{thm:mosaic bimorphisms} it factors uniquely as
 \[
 \begin{tikzcd}
 R \times R \ar[r, "\ast_R"] \ar[d, twoheadrightarrow] & R \\
 R \boxtimes R \ar[ur, dashrightarrow, swap, "m_R"]
 \end{tikzcd}
 \]
where $m_R$ is a morphism of mosaics. Associativity of~$\ast_R$ readily implies associativity of $m_R$, and the identity property of $1 \in R$ similarly implies that $\eta_R$ is a unit for $m_R$. Thus $(R, m_R, \eta_R)$ is indeed an object of $\Mon(\cMsc)$.

Using the universal properties of $\boxtimes$ and $\FF$, it is straightforward to verify that the assignment above forms a functor from multirings to monoid objects in $\cMsc$, which acts identically on morphisms. The fact that this is fully faithful amounts to the following observation for any multirings $R$ and $S$: for a morphism $f \in \Can(R,S) = \cMsc(R,S)$, 
\[
f \circ \ast_R = \ast_S \circ (f \times f) \iff f \circ m_R = m_S \circ (f \boxtimes f). 
\]

Finally, an object $(R,m,\eta)$ of $\Mon(\cMsc)$ is in the essential image of this functor if and only if its underlying additive mosaic is a hypergroup, which happens if and only if the addition is associative by Lemma~\ref{lem:nondegenerate}.
\end{proof}

As hyperrings form a full subcategory $\HRing \subseteq \MRing$ of multirings, the functor above restricts to a fully faithful embedding
\[
\HRing \hookrightarrow \Mon(\cMsc)
\]
as well. The following indicates how to describe the essential image of $\HRing$ under the functor.

\begin{remark}
Let $R$ be a multiring, viewed as a monoid $(R, m, \eta)$ in $\cMsc$. Under the adjunction $R \boxtimes - \dashv \cMsc(R, -)$, the multiplication $m \colon R \boxtimes R \to R$ corresponds to the ``left multiplication'' morphism 
\begin{align*}
\lambda \colon R &\to \cMsc(R,R), \\
r &\mapsto \lambda_r,
\end{align*}
where $\lambda_r(x) = rx$ for all $x \in R$. Similarly, the adjunction $- \boxtimes R \dashv \cMsc(R,-)$ turns $m$ into the ``right multiplication'' morphism $\rho \colon R \to \cMsc(R,R)$. Then $R$ is a hyperring if and only if strict distributivity~\eqref{eq:distributive} holds on both the left and the right, if and only if the morphisms $\lambda$ and $\rho$ have their image in the set of \emph{strict} morphisms
 \[
 \lambda, \rho \colon R \to \cMsc\str(R,R) \subseteq \cMsc(R,R).
 \]
\end{remark} 

\separate

To close this subsection, we describe an alternative way to characterize which objects of $\uHMag$ are mosaics. Recalling the representing objects of Lemma~\ref{lem:representing object}, let us denote 
\begin{align*}
\E &= \E_{\uHMag} = \{e,a,b,c\}, \\
\E_r &= \E_{\Msc}= \{e, a^{\pm 1}, b^{\pm 1}, c^{\pm 1}\}, 
\end{align*}
which are both generated by elements $a$, $b$, and $c$ satisfying $c \in a \star b$. There is a naturally induced injective morphism
\[
\iota \colon \E \hookrightarrow \E_r
\]
that acts identically on the generating objects $a,b,c \in \E$.

\begin{proposition}
Let $M$ be a unital hypermagma and let $\iota \colon \E \hookrightarrow \E_r$ be as above. Then $M$ is reversible (i.e., a mosaic) if and only if, for every morphism $f \colon \E \to M$ there exists a unique morphism $g \colon \E_r \to M$ such that $f = g \circ \iota$ (that is, the map $\iota^* \colon \uHMag(\E_r, M) \to \uHMag(\E, M)$ is bijective).
\end{proposition}

\begin{proof}
First assume that $M$ is reversible, and let $f \colon \E \rightarrow M$ be a morphism of $\uHMag$. Then we have $f(c)\in f(a) \star f(b)$. By reversiblity, we obtain
$f(b)\in f(a)^{-1} \star f(c)$ and $f(a)\in f(c) \star f(b)^{-1}$. One can check that there exists a morphism $g \in \uHMag(\E_r, M)$ defined by $g(x^{\pm 1})=f(x)^{\pm 1}$ for $x=a, b, c$ (here we omit $\iota$ for notational convenience). Since $\E_r$ and $M$ are both reversible, inverses are unique and this condition determines $g$ uniquely.

Conversely, assume that $M$ satisfies bijectivity of $\iota^*$. To study $\uHMag(\E_r, M)$, it will be useful to note that the nontrivial products in $\E_r$ are exactly
\begin{gather*}
c = a \star b, \quad b = a^{-1} \star c, \quad a = c \star b^{-1}, \\
c^{-1} = b^{-1} \star a^{-1}, \quad b^{-1} = c^{-1} \star a, \quad a^{-1} = b \star c^{-1}.
\end{gather*}
We will first prove that every element of $M$ has a unique inverse. For existence, from $x \in e \star x$ we obtain a morphism $f \colon \E \to M$ by $f(a) = e$ and $f(b) = f(c) = x$, which extends to $g \colon \E_r \to M$ by surjectivity of $\iota^*$. Since $b$ is inverse to $b^{-1}$ in $\E_r$, it follows that $x = f(b)$ is inverse to $x_1 = f(b^{-1})$ in $M$. For uniqueness, suppose that $x_2 \in M$ is also an inverse of $x$ in $M$. Then we can construct morphisms $g_i \colon \E_r \to M$ for $i = 1,2$ by setting
\[
g_i(a) = g_i(a^{-1}) = e, \quad g_i(b) = g_i(c) = x, \quad g_i(b^{-1}) = g_i(c^{-1}) = x_i.
\]
Since these agree on $a,b,c \in \E_r$, they satisfy $g_1 \circ \iota = g_2 \circ \iota$. From injectivity of $\iota^*$ we conclude that $g_1 = g_2$ and thus $x_1 = x_2$.

Finally we verify reversibility. Suppose that $x,y,z \in M$ satisfy $x \in y \star z$. Then there is a morphism $f \colon \E \rightarrow M$ such that $f(a)=y, f(b)=z, f(c)=x$. By assumption this extends to a morphism $g \colon \E_r \rightarrow M$,  such that $g(t)=f(t)$ for $t = a,b,c$ (here we suppress the notation of $\iota$). Again $g$ maps inverses to inverses, so by uniqueness of inverses in $M$ we have $g(t^{-1}) = f(t)^{-1}$ for $t = a,b,c$.
Since $b= a^{-1} \star c$ and $a= c \star b^{-1}$ in $\E_r$, we have 
\begin{align*}
z &= g(b) \in g(a^{-1}) \star g(c) = y^{-1} \star x, \\
y &= g(a) \in g(c) \star g(b^{-1}) = x \star z^{-1}.
\end{align*}
This completes the proof.
\end{proof}

\subsection{The category of (canonical) hypergroups}
\label{sub:hypergroups}

We now discuss the categories of hypergroups and canonical hypergroups. 
The difficult lesson to be learned below is that while the objects of these categories are attractive from the algebraic point of view, the categories themselves are not so well-behaved.

\begin{theorem}\label{thm:hypergroup operations}
The categories $\HGrp$ and $\Can$ are closed under the following operations within the category $\uHMag$:
arbitrary products, kernels, strict epimorphic images, and normal epimorphic images.
\end{theorem}

\begin{proof}
Closure under products is straightforward: if $(G_i)_{i \in I}$ are hypergroups, then the unital hypermagma $\prod G_i$ is readily seen to be associative and reversible by a componentwise verification. If the $G_i$ are all commutative, then so is their product.

For the remainder of the proof, let $(G, \cdot, e)$ denote a hypergroup. First suppose $f \in \uHMag(G,M)$. We claim that $\ker f$ is a strict subhypergroup. By Theorem~\ref{thm:normal}, it is a strict unital subhypermagma of $G$, which is associative because $G$ is. It is also closed under inverses by Remark~\ref{rem:absorptive}(6), and thus is a strict subhypergroup.

To see that $\HGrp$ is closed under 
strict epimorphic images, let $p \in \uHMag(G,H)$ be a strict surjective morphism. Then $H$ is the pushout of $\begin{tikzcd}[column sep=0.5cm] H & \ar[l, swap, "p"] G \ar[r, "p"] & H \end{tikzcd}$, so $H$ is a mosaic by Theorem~\ref{thm:reversible forgetful}. Additionally, $H$ is associative by Lemma~\ref{lem:short image} (and if $G$ is commutative, so is $H$). So it follows from Lemma~\ref{lem:nondegenerate} that $H$ is a hypergroup.

Finally, for the case of normal epimorphic images, recall from Theorem~\ref{thm:normal} that normal epimorphisms in $\uHMag$ correspond to unitizations. So consider a unitzation $p \colon G \twoheadrightarrow G_K$ for unital $K \subseteq G$. Without loss of generality, we may assume $K = \ker p$. As described above, $K$ is a strict subhypergroup of $G$.
It follows again by Theorem~\ref{thm:reversible forgetful} that $G_K$ is a mosaic, so we only need to verify associativity thanks to Lemma~\ref{lem:nondegenerate}.
One can readily verify that the equivalence relation on $G$ defined in the proof of Lemma~\ref{lem:unitization}
takes on the form $x \sim y$ if and only if $x \in K y K$. Thus the equivalence classes in $G_K$ are the ``double cosets'' $[x] = KxK$, and one may verify that the operation on double cosets in $G_K$ takes the form
\[
KgK \star KhK = \{KzK \mid z \in gKh\}
\]
for any $g, h \in G$. (The construction of the double coset hypergroup is also described in~\cite[\S 3.3.4]{Zieschang}.) But this product is evidently associative as for any $f,g,h \in G$ we have
\begin{align*}
(KfK \star KgK) \star KhK &= \{KzK \mid z \in fKgKh\} \\
&= KfK \star (KgK \star KhK).
\end{align*}
Finally, if $G$ is commutative then so is $G_K$ (and in fact its elements can be written as ``ordinary cosets'' $gK$ for $g \in G$).
\end{proof}

Unfortunately, (canonical) hypergroups do not have all binary coproducts, as shown in the following example. 
Recall that $\K = \{0,1\}$ denotes the Krasner hyperfield, whose additive structure is determined by $1 + 1 = \{0,1\}$. For any integer $n \geq 1$, we let $\Z_n = \Z/n\Z$ denote the (additive) cyclic group of order~$n$.

\begin{proposition}\label{prop:bad coproduct}
The coproduct $\Z/2\Z \coprod \Z/2\Z$ does not exist in $\HGrp$ or $\Can$. 
\end{proposition}

\begin{proof}
We give the proof within the category $\Can$, with the case of $\HGrp$ being very similar since $\Z_2$ and $\K$ are objects in the full subcategory $\Can$.
Assume for contradiction that the coproduct $G = \Z_2 \coprod \Z_2$ exists in $\Can$. Let $\tau \colon \Z_2 \to \K$ be the morphism of hypergroups (even hyperfields) that acts as the identity on the underlying set, and note that $\Can(\Z_2,\K) = \{\tau, 0\}$. 
Then
\[
\Can(\Z_2 \coprod \Z_2, \K) \cong \Can(\Z_2, \K) \times \Can(\Z_2, \K).
\]
has four elements, given by $0 = 0 \coprod 0$, $\tau_1 = \tau \coprod 0$, $\tau_2 = 0 \coprod \tau_2$, and $\tau_{12} = \tau \coprod \tau$.
Proposition~\ref{prop:strict classifier} implies
\[
\Sub\str(G) \cong \Can(G,\K),
\]
so that $G$ has four strict subhypergroups. Two of these are given by $G = \ker 0$ and $0 = \ker \tau_{12}$, and the proper nontrivial subgroups are $\ker \tau_i$ for $i = 1,2$.
Let $x = i_1(1)$ and $y = i_2(1)$ denote the image of $1 \in \Z_2$ under the two structure maps $i_1, i_2 \colon \Z_2 \to G$. Then we have
\[
x \notin \ker \tau_1, \ x \in \ker \tau_2, \ y \in \ker \tau_1, \ y \notin \ker \tau_2.
\]
As a consequence, the four subgroups of $G$ are uniquely determined by their intersection with $\{x,y\}$.

Let $\phi = \id_{\Z_2} \coprod \id_{\Z_2} \colon G \to \Z_2$. Because $\ker \phi$ contains neither $x$ nor $y$, it follows that $\ker \phi = 0$. On the other hand,
\[
\phi(x+y) \subseteq \phi(x) + \phi(y) = 1 + 1 = 0
\]
So $x+y \subseteq \ker \phi$, which implies that $x+y = 0$. 
It follows that $y = -x$, so that $x$ and $y$ are contained in the same strict subhypergroups of $G$. But this contradicts, for instance, the claim above that $x \notin \ker \tau_1$ while $y = -x \in \ker \tau_1$.
\end{proof}

Next we present an example to show that $\Can$ does not have all equalizers.	Let $\F_q$ denote the field with~$q$ elements, considered as an abelian group under addition. Below we consider the quotient $\F_9/\F_3^\times$ hyperfield as in~\cite{Krasner:quotient}, whose underlying canonical hypergroup is of the form described in Example~\ref{ex:group action}.

\begin{proposition}\label{prop:bad equalizer}
Consider the quotient hypergroup $H = \F_9/\F_3^\times$, and let $F \colon H \to H$ be the morphism induced by the Frobenius automorphism $x \mapsto x^3$ on $\F_9$. Then there is no equalizer of the morphisms
\[
\begin{tikzcd}
H \ar[r, shift left=.5ex, "\id_H"] \ar[r, shift right=.5ex, swap, "F"] & H
\end{tikzcd}
\]
in the category $\HGrp$ or $\Can$.
\end{proposition}

\begin{proof}
Let $\alpha \in \F_9^\times$ be a generator of the multiplicative group of the field with nine elements. Then the elements of the quotient hyperfield $H$ are of the form $[0]$ and $[\alpha^i] = [\alpha]^i$ for $i = 0, \dots, 3$.

As before we write the proof in the case of $\Can$, while the proof for $\HGrp$ is essentially identical.
Assume for contradiction that the equalizer $e \colon E \to H$ exists in $\Can$. 
By construction as the equalizer of the identity and Frobenius maps, every element $\beta$ of the image of $e$ must satisfy $\beta^3 = \beta$ under the hyperfield multiplicative structure. 

On the other hand, consider the two morphisms $f,g \colon \K \to H$ given by $f(1) = [1]$ and $g(1) = [\alpha^2]$. Then $f$ and $g$ both equalize $\id_H$ and $F$ (since $\alpha^2$ is identified with $(\alpha^2)^3 = -\alpha^2$ in $H$), so that they factor through $e$. In particular, there exist $x,y \in E$ such that $e(x) = [1]$ and $e(y) = [\alpha^2]$. By the assumption that $E$ is a canonical hypergroup, we have $x + y \neq \varnothing$. Fixing any $z \in x + y$, we find that
\[
e(z) \in e(x+y) \subseteq e(x) + e(y) = [1] + [\alpha^2] = \{[\alpha], [\alpha^3]\}.
\]
However, this contradicts the requirement that $e(z) = e(z)^3$ described above. Thus the equalizer $E$ does not exist.
\end{proof}

Furthermore, $\Can$ does not have arbitrary coequalizers. 
The proof makes use of the following fact. An element $g$ of a hypergroup $G$ is defined to be \emph{scalar}~\cite[p.~299]{Roth:character} if, for all $x \in G$, the products $g \star x$ and $x \star g$ are singletons.

\begin{lemma}\label{lem:scalar}
Let $G$ be a hypergroup, and let $g \in G$. Then $g$ is scalar if and only if $g \star g^{-1} = \{e\} = g^{-1} \star g$.
The set of scalar elements of $G$ forms a strict subhypergroup that is in fact a group.
\end{lemma}

\begin{proof}
The equivalence of these two conditions is proved, for instance, in~\cite[Lemma~1.6]{Roth:character} or~\cite[Lemma~1.4.3]{Zieschang}. 
To prove the last sentence, note first that scalar elements include the identity and are closed under formation of inverses. It is then straightforward to check (using associativity) that if $g,h \in G$ are scalar, then the single element in $gh$ is also scalar. (This is also proved in~\cite[Lemma~1.4.3(iii)]{Zieschang}.) So they form a strict subhypergroup, and by the scalar property each of the products in this subhypergroup is a singleton. Thus we obtain a group.
\end{proof}

\begin{proposition}\label{prop:bad coequalizer}
In the abelian group $M = \Z_3 \oplus \Z_2$, denote $e_1 = (1,0)$ and $e_2 = (0,1)$. The group homorphisms 
\[
\begin{tikzcd}
\Z \ar[r, shift left=.5ex] \ar[r, shift right=.5ex, swap] & M
\end{tikzcd}
\]
given by $1 \mapsto e_1 + e_2$ and $1 \mapsto 2e_1 + e_2$ have no coequalizer in either of the categories $\Can$ or $\HGrp$.
\end{proposition}

\begin{proof}
Assume toward a contradiction that there is a coequalizer $p \colon M \to C$ in $\HGrp$ (the argument for $\Can$ being identical). We denote its hyperoperation and identity as $(C, \star,e)$ because we are not (yet) guaranteed that is hyperoperation is commutative. Below we will show that the image of $p$ is a strict subhypergroup of $C$ that in fact forms a group of order~4 or~5. This will mean that $p \colon M \to p(M)$ is a group homomorphism. But $|M| = 6$ is not divisible by~4 or~5, yielding a contradiction.

We first claim that the only strict subhypergroup of $C$ that does not contain either $p(e_1)$ or $p(e_2)$ is $\{0\}$.
Combining Proposition~\ref{prop:strict classifier} with the universal property of $C$, we find that there are bijections
\begin{align*}
\Sub\str(C) &\cong \Hom(C,\K) \\
&\cong \{g \in \Hom(M,\K) \mid g(e_1 + e_2) = g(2e_1 + e_2)\} \\
&\cong \{N \in \Sub\str(M) \mid e_1 + e_2, 2e_1 + e_2 \in N \mbox{ or } e_1 + e_2, 2e_1 + e_2 \notin N\} \\
&= \{0, \Z e_1, \Z e_2, M\}.
\end{align*}
(Note that a strict subhypergroup of $M$ is merely a subgroup.)
This correspondence sends each of the subgroups $N \in \{0, \Z e_1, \Z e_2, V\}$ to the strict subhypergroup $p^{-1}(N) = \ker(\chi \circ p) \subseteq C$, where $\chi \colon M \to \K$ is the morphism uniquely determined by $\ker \chi = N$. In particular, we may verify from this that the only strict subhypergroup of $C$ that does not contain either $p(e_1)$ or $p(e_2)$ is the trivial one.

We now consider two morphisms from $M$ to different canonical hypergroups $M_1$ and $M_2$ of order~4. The first hypergroup $M_1 = M/\{ \pm 1 \}$ is the quotient (as in Example~\eqref{ex:group action}) of $M$ by the group $\{\pm 1\}$ acting by negation. The second hypergroup $M_2$ is a quotient mosaic of $M$ by an equivalence relation satisfying $e_1 + e_2 \sim 2e_1 + e_2 \sim e_2$, which happens to be a canonical hypergroup. (This hypergroup is denoted $H_{4,2}$ in the classification of~\cite[\S 7.2]{Zieschang}.) The addition tables for the nonzero elements of these hypergroups are given below:
\[
\begin{array}{|c|c|c|c|}
\hline
M_1 & x_1 & y_1 & z_1 \\ 
\hline
x_1 & 0,x_1 & z_1 & y_1, z_1 \\
\hline
y_1 & z_1 & 0 & x_1 \\
\hline
z_1 & y_1, z_1 & x_1 & 0, x_1 \\ \hline
\end{array}
\hspace{2cm}
\begin{array}{|c|c|c|c|}
\hline
M_2 & x_2 & y_2 & z_2 \\ 
\hline
x_2 & z_2 & y_2 & 0 \\
\hline
y_2 & y_2 & 0,x_2, z_2 & y_2 \\
\hline
z_2 & 0 & y_2 & x_2 \\ \hline
\end{array}
\]
The morphisms $f_i \colon M \to M_i$ are each uniquely determined by the values
\[
f_i(e_1) = x_i, \quad f_i(e_2) = y_i. 
\]
The morphism $f_1 \colon M \to M_1 = M/\{\pm\}$ is simply the quotient map, so that $f_1(-e_1) = f_1(e_1) = x_1$. From this it follows that
\[
f_1(e_1 + e_2) = f_1(2e_1 + e_2) = x_1 + y_1 = z_1.
\]
On the other hand, $f_2(-e_1) = -f_2(e_1) = z_2$. In this case we can similarly verify that
\[
f_2(e_1 + e_2) = f_2(2e_1 + e_2) = y_2.
\]

It follows from the universal property of the coequalizer that each of these morphisms factors as $f_i = \phi_i \circ p$ for a morphism of hypergroups $\phi_i \colon C \to M_i$. Note that because each $\phi_i(p(e_j)) = f_i(e_j) \neq 0$, it follows from the above description of the strict subhypergroups of $M$ that $\ker \phi_i = \{0\}$.

We will finally show that $p(M) \subseteq C$ is a strict subhypergroup of order~4 or~5, which is in fact an abelian group.  Regarding the order of $p(M)$, notice that because $f_1$ and $f_2$ separate the elements $0, e_1, e_2, e_1+e_2$, their images under $p$ must be distinct, so that $|p(M)| \geq 4$. But $C$ is constructed so that $p$ identifies $e_1 + e_2$ with $2e_2 + e_2$, so that $|p(M)| \leq 5$.
To verify that $p(C)$ is a group, it will suffice by Lemma~\ref{lem:scalar} to show that its elements are scalar. First notice that 
\begin{align*}
\phi_2(p(e_1) \star p(e_1)^{-1}) &\subseteq f_2(e_1) - f_2(e_1) = x_2 + z_2 = \{0\}, \\
\phi_1(p(e_2) \star p(e_2)^{-1}) &\subseteq f_1(e_2) - f_1(e_2) = y_1 + y_1 = \{0\}.
\end{align*}
Since $\phi_1$ and $\phi_2$ have trivial kernels, we see that $p(e_1) \star p(e_1)^{-1} = \{e\} = p(e_2) \star p(e_2)^{-1}$. 
Because the $M_i$ are commutative, we similarly have $p(e_i)^{-1} \star p(e_i) = \{e\}$ for $i = 1,2$, so that both $p(e_i)$ are scalar elements.

Because $M$ is generated by $e_1$ and $e_2$, the image $p(C)$ is contained in the strict subhypergroup $S \subseteq C$ generated by $p(e_1)$ and $p(e_2)$. But this $S$ consists of scalar elements by Lemma~\ref{lem:scalar}, so that $p(C)$ contains scalar elements. Thus we have shown $p(C)$ is a group, completing the proof.
\end{proof}

\separate

Now we will show that $\Can$ is not closed under formation of the internal hom of $\cMsc \subseteq \uHMag$. This will rely on the observation that within the category $\Can$, the (hyper)group $\Z_2$ represents the following functor
\begin{equation}\label{eq:F2 represents}
\Can(\Z_2, G) \cong \{x \in G \mid 0 \in x + x\},
\end{equation}
where $f \in \Can(\Z_2, G)$ corresponds to the element $x =f(1) \in G$. Under this bijection, a sum of morphisms $f + g$ corresponds to the set $\{x \in f(1) + g(1) \mid 0 \in x + x\}$. Thus to show that there exist $f, g \in \Can(\Z_2, G)$ with $f + g = \varnothing$, it suffices to find elements $x,y \in M$ such that
\[
0 \in x + x, \ 0 \in y + y, \quad \mbox{but} \quad 0 \notin z + z \mbox{ for all } z \in x + y.
\]

To realize this strategy, consider the following commutative ring given by generators and relations:
\[
R = \Z[x,y,z \mid 2x = y(z+1) = 0, \, z^2 = 1]
\]
Within this ring, $G = \{1, z\}$ forms a multiplicative group of order two. We may then form the quotient hyperring $R/G$ as in~\cite{ConnesConsani:adele}, which has hyperaddition given by
\begin{align*}
[f] + [g] &= \{[h] \mid h \in fG + gG\} \\
&= \{[f+g], [fz + g], [f+gz], [fz+gz]\} \\
&= \{[f+g], [fz+g] \}.
\end{align*}

\begin{theorem}\label{thm:empty sum}
Let $f, g \in \Can(\Z_2, R/G)$ be the morphisms given by $f(1) = [x]$ and $g(1) = [y]$. Then $f + g = \varnothing$; in particular, the commutative unital hypermagma $\Can(\Z_2, R/G)$ is not a hypergroup.
\end{theorem}

\begin{proof}
Because $0 = y(z+1) = yz + y \in yG + yG$, we have $[0] \in [y] + [y]$; an even easier argument gives $[0] \in [x] + [x]$. Thus the formulas given for $f$ and $g$ indeed describe morphisms in $\Can(\Z_2, R/G)$.

As described via~\eqref{eq:F2 represents}, to see that $f + g = \varnothing$ in $\Can(\Z_2, R/G)$ it is enough to verify that for all $[w] \in [x] + [y]$, we have $0 \notin [w] + [w]$. By construction of the quotient hyperring $R/G$ we have
\[
[x] + [y] = \{[x+y], [xz + y]\}.
\]
We verify that
\begin{align*}
[x+y] + [x+y] &= \{[x+y+x+y], [(x+y)z+x+y]\} \\
	&= \{[2y], [x(z+1)]\}, \\
[xz+y] + [xz+y] &= \{[xz+y+xz+y], [xz+y+(xz+y)z]\} \\
	&= \{[2y], [x(z+1)]\}. 
\end{align*}
It is straightforward to check that $2y, x(z+1) \neq 0$ in $R$, so that $[0] \notin [w] + [w]$ for both values of $[w] \in [x] + [y]$. This completes the proof.
\end{proof}

Thus canonical hypergroups do not naturally inherit a closed monoidal structure in the way that commutative mosaics do. Yet there is even stronger evidence for the lack of a ``tensor product'' structure on $\Can$, in the form of the following example which forbids the existence of monoidal products that represent bimorphisms. 

\begin{theorem}\label{thm:Klein four}
Let $V = \Z/2\Z \times \Z/2\Z$ denote the Klein four-group. Then the functor of bimorphisms
\[
\Bim_{\Can}(V,V; -) \colon \Can \to \Set
\]
is not representable.
\end{theorem}

\begin{proof}
We denote the nonzero elements of this group by $V \setminus \{0\} = \{a_1, a_2, a_3\}$. Fix a canonical hypergroup $M$, and suppose that $B \colon V \times V \to M$ is a bimorphism. Denote the elements $x_{ij} := B(a_i, a_j) \in M$, which form a matrix $(x_{ij})_{i,j = 1}^3$. Applying $B(-,a_j)$ to $a_1 = a_2 + a_3$ and $a_i + a_i = 0$ gives 
\[
x_{1j} \in x_{2j}+x_{3j} \quad \mbox{and} \quad 0 \in x_{ij} + x_{ij}. 
\]
Similarly, applying $B(-,a_j)$ to the equations $a_{\sigma(1)} = a_{\sigma(2)} + a_{\sigma(3)}$ for all permutations $\sigma$ of $\{1,2,3\}$ yields permuted relations $x_{\sigma(1)j} \in x_{\sigma(2)j}+x_{\sigma(3)j}$. However, these automatically follow from the original relation (where $\sigma$ is the identity) and $x_{ij} = -x_{ij}$. Applying similar reasoning with the morphisms $B(a_i,-)$, we find that there is a natural isomorphism
\begin{multline}\label{eq:matrix representation}
\Bim(V,V; M) \cong \{(x_{ij})_{i,j = 1}^3 \mid x_{ij} \in M \mbox{ satisfy } x_{ij} = -x_{ij},\\
 x_{1j} \in x_{2j}+x_{3j}, \  
\mbox{and } x_{i1} \in x_{i2}+x_{i3} \}.
\end{multline}
given by the assignment $B \mapsto (B(a_i,a_j))_{i,j}$.

Assume toward a contradiction that $\Bim(V,V; -)$ is represented by a canonical hypergroup $V \otimes V$. Let $a_i \otimes a_j \in V \otimes V$ denote the image of $(a_i, a_j) \in V \times V$ under the bimorphism given by the Yoneda element
\[
\id_{V \otimes V} \in \Can(V \otimes V, V \otimes V) \cong \Bim(V \times V; V \otimes V).
\]
Because a bimorphism $V \times V \to M$ is determined by its values on the $(a_i, a_j) \in V \times V$, one can deduce that a morphism $V \otimes V \to M$ is uniquely determined by the images of the $a_i \otimes a_j \in V \otimes V$. 

By Proposition~\ref{prop:strict classifier} we have a bijection 
\[
\Sub\str(V \otimes V) \cong \Can(V \otimes V, \K), 
\] 
where the set on the right is in turn described by matrices as in~\eqref{eq:matrix representation}. One such matrix is
$\left(\begin{smallmatrix} 0 & 0 & 0 \\ 0 & 1 & 1 \\ 0 & 1 & 1 \end{smallmatrix} \right)$, which corresponds to a morphism whose kernel contains $a_1 \otimes a_1$ but not $a_2 \otimes a_3$. In particular, we deduce that 
\[
a_1 \otimes a_1 \neq a_2 \otimes a_3.
\]
Furthermore, there is a unique such matrix with no zero entries, namely $\left(\begin{smallmatrix} 1 & 1 & 1 \\ 1 & 1 & 1 \\ 1 & 1 & 1 \end{smallmatrix} \right)$. 
Because this matrix corresponds to the zero subhypergroup of $V \otimes V$, we find that zero is the only strict subhypergroup containing none of the $a_i \otimes a_j$. In particular, a morphism out of $V \otimes V$ has kernel zero if all $a_i \otimes a_j$ map to nonzero elements.

Finally, consider the morphism $\phi \colon V \otimes V \to V$ that corresponds to the matrix $\left(\begin{smallmatrix} a_1 & a_2 & a_3 \\ a_2 & a_3 & a_1 \\ a_3 & a_1 & a_2 \end{smallmatrix} \right)$ which satisfies~\eqref{eq:matrix representation}. Since none of the $a_i \otimes a_j$ are in the kernel of $\phi$, it follows as above that its kernel is zero. Then because
\[
\phi(a_1 \otimes a_1 - a_2 \otimes a_3) \subseteq \phi(a_1 \otimes a_1) - \phi(a_2 \otimes a_3) = a_1 - a_1 = 0,
\]
we must have $a_1 \otimes a_1 - a_2 \otimes a_3 = \{0\}$. By uniqueness of inverses in canonical hypergroups, we derive the contradiction $a_1 \otimes a_1 = a_2 \otimes a_3$. 
\end{proof}

\subsection{Matroids as mosaics}

We now show how matroids are able to provide examples of mosaics that are not necessarily associative, and thus need not be hypergroups. This construction works for infinite matroids. It was shown in~\cite{BDKPW:matroids} that various axiom systems for infinite matroids are equivalent if one includes a certain maximality condition, which is automatically satisfied for finitary matroids. For our purposes it will be most convenient to work with the definition via closure operators, but as the maximality condition does not seem to be relevant to the construction of this functor we do not include it. While we do not make use of bases or rank below, we mention that the notions of \emph{finite} independent sets and \emph{finite} rank still behave well in this setting~\cite[Chapter~4]{FaureFrolicher}.

A \emph{closure operator}~\cite[Section~3.1]{FaureFrolicher} on a set $M$ is a map $C \colon \P(M) \to \P(M)$ that is 
\begin{itemize}
\item \emph{extensive:} $S \subseteq C(S)$,
\item \emph{monotone:} $S \subseteq T \implies C(S) \subseteq C(T)$, and
\item \emph{idempotent:} $C(C(S)) = C(S)$,
\end{itemize}
where $S,T \subseteq M$ are arbitrary subsets. A subset $S$ of $M$ is \emph{closed (with respect to $C$)} if $C(S) = S$.
A \emph{closure space} $(X,C)$ is a set $X$ with a closure operator $C = C_X$ on $X$. We define the category $\Clos$ of closure spaces to have morphisms $f \colon (X, C_X) \to (Y, C_Y)$ given by functions $f \colon X \to Y$ satisfying 
\[
f(C_X(S)) \subseteq C_Y(f(S))
\]
for all $S \subseteq X$, or equivalently, if the preimage of every closed subset of $Y$ under $f$ is again closed in $X$.

A \emph{matroid} is closure space $M$ whose closure operator $C$ satisfies the \emph{exchange property}: for $x,y \in M$ and $S \subseteq M$,
\[
x \notin C(S) \mbox{ and } x \in C(S \cup y) \implies y \in C(S \cup x).
\]
Closed subsets of a matroid $M$ are also called \emph{subspaces} (or \emph{flats}) of $M$.
We let $\Mat$ denote the full subcategory of $\Clos$ whose objects are the matroids. In the literature on matroids, the morphisms of this category are called \emph{strong maps}.

We pause to recall some necessary terminology about matroids.
Let $M$ be a matroid. 
An element of $C(\varnothing)$ (if any exists) is called a \emph{loop} of $M$; equivalently, a loop is an element contained in every closed subset of $M$. 
Two non-loop elements $x,y \in M$ are \emph{parallel} if $x \neq y$ and $x \in C(y)$, or equivalently (by the exchange axiom), $y \in C(x)$. 

A matroid $M$ is \emph{simple} if its closure operator satisfies
\[
C(\varnothing) = \varnothing \mbox{ and } C(x) = \{x\} \mbox{ for all } x \in M,
\]
or equivalently, if it has no loops or parallel elements. 
A \emph{pointed matroid} $(M,0)$ is a matroid with a choice of distiguished loop $0 = 0_M \in M$.  We similarly say that a pointed matroid $(M,0)$ is \emph{simple} if it satisfies
\[
C(\varnothing) = \{0\} \mbox{ and } C(x) = \{x, 0\} \mbox{ for all } x \in M,
\]
or equivalently, if it has no loops other than $0$ and no parallel elements.

Several related categories of matroids are defined as follows:
\begin{itemize}
\item $\Mat_\bullet$ is the category of pointed matroids with strong morphisms that preserve distinguished loops;
\item $\sMat$ is the full subcategory of $\Mat$ consisting of the simple matroids;
\item $\sMat_\bullet$ is the full subcategory of $\Mat_\bullet$ consisting of the simple pointed matroids.
\end{itemize}
These categories and several others were studied in the case of finite matroids in~\cite{HeunenPatta:matroids}.

Given a pointed simple matroid $(M,0)$, we define a commutative hyperoperation $+ \colon M \times M \to \P(M)$ by setting $0$ to be the additive identity and, for all $x,y \in M \setminus \{0\}$,
\[
x + y = \begin{cases}
C(x,y) \setminus \{x,y,0\}, & x \neq y \\
\{x,0\}, & x = y.
\end{cases}
\]
We note immediately that for all $x,y \in M$, this hyperaddition satisfies
\[
x + y \subseteq C(x,y).
\]
This is a slight adjustment of a similar hyperoperation defined in the context of projective geometries in~\cite{Prenowitz}, which was stated in terms of closure operators in~\cite[Proposition~3.3.4]{FaureFrolicher}. It was already recognized in~\cite[Proposition~3.1]{ConnesConsani:adele} that the above hyperoperation forms a hypergroup for many projective geometries. We will return to this connection with projective geometries below.

\begin{theorem}\label{thm:matroid mosaics}
If $(M,0)$ is a pointed simple matroid, then $(M,+)$ is a commutative mosaic satisfying $-x = x$ for all $x \in M$. The assignment $(M,0) \mapsto (M,+)$ which acts identically on morphisms determines a faithful functor
\[
\sMat_\bullet \longrightarrow \cMsc.
\]
\end{theorem}

\begin{proof} 
Since $0 \in x + x$ for all $x \in M$, we define $-x = x$. Suppose $x \in y + z$ is satisfied in $M$. If either $y$ or $z$ is $0$, then it is straightforward to verify that $y \in z - x$. So assume $y,z \neq 0$. 
\begin{description}
\item[Case $y = z$] Here we have $x \in y + y = \{y,0\}$. If $x = 0$ then $y \in y + 0 = z - x$, and if $x = y$ then $y \in y + y = z - x$. 
\item[Case $y \neq z$]
In this case we must have $x \notin \{y,z,0\} = C(y) \cup C(z)$ by definition of $y+z$. 
Then by the exchange property,
\[
x \in C(\{z\} \cup \{y\}) \setminus C(z) \implies y \in C(\{z\} \cup \{x\}) \setminus C(z).
\]
Since $y \neq x$ and $C(z) = \{z,0\}$, it follows that $y \in C(z,x) \setminus \{z,x,0\} = z - x$. 
\end{description}
Thus reversibility is satisfied in all cases, and $M$ is a commutative mosaic. 

Now suppose $f \colon M \to N$ is a morphism in $\sMat_\bullet$. We claim that the same function is a morphism of mosaics $(M,+) \to (N,+)$. Because $f$ preserves loops, it satisfies $f(0_M) = 0_N$. Now let $x,y \in M$; we wish to show that $f(x+y) \subseteq f(x) + f(y)$. This is trivially satisfied if either $x$ or $y$ is $0_M$, so we may assume $x,y \neq 0$. If $x = y$ then
\[
f(x + x) = f(\{0,x\}) = \{0, f(x)\} = f(x) + f(x).
\]
So we may assume that $x$ and $y$ are distinct and nonzero. In this case we have
\begin{align*}
f(x+y) = f(C(x,y) \setminus \{x,y\}) \subseteq C(f(x),f(y)).
\end{align*}
First suppose that $f(x) = 0$.
If also $f(y) = 0$ then we have $f(x+y) = 0 = f(x) + f(y)$. So assume $f(y) \neq 0$. Let $z \in x + y$, so that $y \in z - x = x + z$. Then
\[
f(y) \in f(x + z) \subseteq f(C(x,z)) \subseteq C(f(x), f(z)) = C(0,f(z)) = \{0, f(z)\}.
\]
Since $f(y) \neq 0$ we have $f(y) = f(z)$. This shows that 
\[
f(x+y) \subseteq \{f(y)\} = 0 + f(y) = f(x) + f(y).
\]
A symmetric argument applies if $f(y) = 0$. Next suppose that $f(x) = f(y) \neq 0$. Then
\[
f(x+y) \subseteq C(f(x)) = \{0,f(x)\} = f(x) + f(x) = f(x) + f(y)
\]
Finally, we may assume that $f(x)$ and $f(y)$ are distinct and nonzero. Then 
\[
f(x) + f(y) = C(f(x),f(y)) \setminus\{f(x),f(y), 0\}. 
\]
In this case it remains to check that if $z \in x+y$ then $f(z) \notin \{f(x), f(y), 0\}$.  
Since $z \in x + y$ we have $y \in z - x = x + z \subseteq C(x,z)$, from which it follows that 
\[
f(y) \in f(C(x,y)) \subseteq C(f(x),f(z)).
\] 
If $f(z)$ is equal to either $f(x)$ or $0$, then we have 
\[
f(y) \in C(f(x), f(z)) \subseteq C(f(x), 0) = \{f(x),0\},
\]
which contradicts our assumption on $f(x)$ and $f(y)$. Similarly, if we assume $f(z) = f(x)$ then we may deduce the contradiction $f(x) \in \{f(y),0\}$. This proves $f(z) \notin \{f(x),f(y),0\}$ as desried.

Thus we have verified in all cases that $f$ is a morphism of mosaics. 
It follows readily that we obtain a functor $\sMat_\bullet \to \cMsc$ which commutes with the forgetful functors to $\Set$. Since these forgetful functors are faithful, the functor just constructed is also faithful.
\end{proof}

Note that the mosaic associated to a simple pointed matroid $(M,0)$ is not associative if there exist empty sums. While this follows from Lemma~\ref{lem:nondegenerate}, we may also directly verify that if $x + y = \varnothing$, then
\[
(x + x) + y = \{0,x\} + y = y \neq \varnothing = x + (x + y).
\]
Thus we cannot corestrict the functor above to have codomain $\Can$. We require the generality of empty and nonassociative sums in order to define the hyperstructure.

The functor above can be used to define functors from several related categories of matroids by composing with various other functors (which are described in the finite case in subsections~4.1, 4.3, and~7.1 of~\cite{HeunenPatta:matroids}). First note that there is a faithful functor $\Mat \to \Mat_\bullet$ from matroids to pointed matroids, denoted on objects by $M \mapsto M_0 := M \sqcup \{0\}$,  that freely adjoins a loop to each matroid. It sends simple matroids to simple pointed matroids, thus restricting to a faithful functor $\sMat \to \sMat_\bullet$. 

Furthermore, there are \emph{simplification} $\si \colon \Mat \to \sMat$ and \emph{pointed simplification} $\si_\bullet \colon \Mat_\bullet \to \sMat_\bullet$ functors as follows. For any closure space $X$, the set $L(X)$ of closed subsets of $X$ forms a complete lattice. Conversely, if $L$ is a complete atomistic lattice, form a closure space $G(L)$ whose underlying set is the set of atoms of $L$, and with closure operator given by $C(S) = \{a \in G(L) \mid a \leq \bigvee S\}$; by construction one can verify that the smallest nonempty closed sets in $G(L)$ are the singletons consisting of the atoms. Then it follows from~\cite[Proposition~3.4.9]{FaureFrolicher} that:
\begin{itemize}
\item If $M$ is a matroid then $L(M)$ is complete, atomistic, and semimodular;
\item If $L$ is complete, atomistic, and semimodular, then $G(L)$ is a simple matroid.
\end{itemize}
Then the assignment $M \mapsto G(L(M))$ yields a functor $\si \colon \Mat \to \sMat$ as in~\cite[Definition~7.11]{HeunenPatta:matroids}. In the case where $(M,0)$ is a pointed matroid, we take $\si_\bullet(M,0) = \si(M)_0$ and one verifies that we obtain a functor $\Mat_\bullet \to \sMat_\bullet$ as in~\cite[Theorem~7.12]{HeunenPatta:matroids}. One can verify that these simplification functors are respectively left adjoint to the forgetful functors $\sMat \to \Mat$ and $\sMat_\bullet \to \Mat_\bullet$.

We obtain the following commuting diagram of functors, with horizontal arrows being faithful:
\[
\begin{tikzcd}
\sMat \ar[r, hookrightarrow] & \sMat_\bullet \ar[r, hookrightarrow] & \cMsc \\
\Mat \ar[u] \ar[r, hookrightarrow] & \Mat_\bullet \ar[u] & 
\end{tikzcd}
\]
In this way we can relate all four categories of matroids above to commutative mosaics by composing.

\separate

From the construction of the mosaic structure $(M,+)$ on a pointed matroid $(M,0)$, the functor in Theorem~\ref{thm:matroid mosaics} cannot be full. This is because the hyperoperation on $M$ encodes the action of the closure operator $C_M$ on sets of cardinality at most two (since $x + y \subseteq C(x,y)$). If one hopes to obtain a full functor, it is then sensible to restrict to matroids whose closure operators are, in some sense, fully determined by the closure of two-element sets. We will verify this intuition in Theorem~\ref{thm:projective embedding} below. Because the corresponding matroids arise from projective geometries, we pause to define these geometries and describe the connection. We follow the definitions of projective geometries and their morphisms as developed by Faure and Fr\"{o}licher in~\cite{FaureFrolicher:morphisms, FaureFrolicher:categorical}. Several equivalent descriptions of the structure of a projective geometry can be found in~\cite[Ch.~2--3]{FaureFrolicher}; the one given below is taken from~\cite[Exercise~2.8.1]{FaureFrolicher}.

A \emph{projective geometry} is a set $G$ equipped with a collection $\Delta \subseteq \P(G)$ of \emph{lines}, subject to the axioms:
\begin{itemize}
\item Each line contains at least two points;
\item Any two distinct points $a \neq b$ lie on a unique line $\overline{ab}$;
\item If $a,b,c,d$ are four distinct points and the lines $\overline{ab}$ and $\overline{cd}$ intersect, then so do $\overline{ac}$ and $\overline{bd}$.
\end{itemize}
A \emph{subspace} $E$ of $G$ is a subset such that, if $a,b \in E$ and $a \neq b$, then $\overline{ab} \subseteq E$. There is an associated hyperoperation $\star$ on $G$ that is defined by letting $a \star b$ be the smallest subspace of $G$ containing $a$ and $b$. Thus 
\[
a \star b = \begin{cases}
\overline{ab}, & a \neq b, \\
\{a\}, & a = b.
\end{cases}
\]
(One can alternatively define projective geometries in terms of such an operation, as in~\cite[Section~2.2]{FaureFrolicher}.)

Morphisms of projective geometries are certain partially defined functions. A \emph{partial function} $f \colon X \dashrightarrow Y$ is a function $f \colon \dom f \to Y$ defined on a subset $\dom f \subseteq X$. The complement of its domain is the \emph{kernel} of $f$, denoted $\ker f = X \setminus \dom f$. Any partial function induces a ``preimage'' mapping $f^\sharp \colon \P(Y) \to \P(X)$ by defining, for $E \subseteq Y$,
\[
f^\sharp(E) = f^{-1}(E) \cup \ker f.
\]
If $g \colon Y \dashrightarrow Z$ is another partial function, the composite $g \circ f \dashrightarrow X \to Z$ is defined by setting $\ker (g \circ f) = f^\sharp(\ker g)$ and defining it to be the composite function on the complement of the kernel. In this way we obtain the category $\Par$ whose objects are sets and whose morphisms are partial functions. There is a straightforward equivalence of categories~\cite[Proposition~6.1.18]{FaureFrolicher}
\[
\begin{tikzcd}
\Par \ar[r, "\sim"] & \Set_\bullet,
\end{tikzcd}
\]
which acts on objects by $X \mapsto X_0 := (X \sqcup \{0\}, 0)$ and which sends $f \in \Par(X,Y)$ to the function $f_0 \colon X_0 \to Y_0$ that extends $f \colon X \setminus \ker f \to Y \subseteq Y_0$ to $X_0$ by mapping $\ker f \sqcup \{0\}$ to $0 \in Y_0$. This evidently has the property that 
\[
f_0^{-1}(E) = f^\sharp(E) \sqcup \{0\}
\]
for all $E \subseteq Y$.

A \emph{morphism} $f \colon G \dashrightarrow H$ of projective geometries is a partially defined function such that:
\begin{itemize}
\item $\ker f$ is a subspace of $G$;
\item For $a,b \in G \setminus \ker f$ and $c \in \ker f$, if $a \in \overline{bc}$ then $f(a) = f(b)$;
\item If $a,b,c \in G \setminus \ker f$ with $a \in b \star c$, then $f(a) \in f(b) \star f(c)$.
\end{itemize}
By~\cite[Proposition~6.2.3]{FaureFrolicher}, a partial function $f \colon G \dashrightarrow H$ is a morphism of projective geometries if and only if $f^\sharp(E)$ is a subspace of $G$ whenever $E$ is a subspace of $H$.
These morphisms are  preserved under composition of partial functions. 
In this way projective geometries and their morphisms form a category, denoted by $\Proj$, with a faithful forgetful functor $\Proj \to \Par$.

The structure of a projective geometry on a set was shown to be equivalent to a certain matroid structure in~\cite[Section~3.3]{FaureFrolicher}, which we recall here.
Let $M$ be a matroid. We say that:
\begin{itemize} 
\item $M$  is \emph{finitary} if, for all $S \subseteq M$,
\[
C(S) = \bigcup \, \{C(T) \mid T \subseteq S \mbox{ is finite}\}.
\]
\item $M$ satisfies the \emph{projective law} if, for all $S,T \subseteq M$,
\[
C(S \cup T) = \bigcup \{C(x,y) \mid x \in C(S), y \in C(T)\}.
\]
\end{itemize}
Finitary simple matroids have been called \emph{(combinatorial) geometries}~\cite{CrapoRota}. 

We define a \emph{projective (pointed) matroid} to be a finitary simple (pointed) matroid that satisfies the projective law. Let $\pMat$ and $\pMat_\bullet$ respectively denote the full subcategories of $\sMat$ and $\sMat_\bullet$ consisting of the (pointed) projective matroids. 

Let $G$ be a projective geometry, and let $C$ be the closure operator on $G$ that assigns to $S \subseteq G$ the smallest subspace $C(S) \subseteq G$ containing $S$. It is shown in~\cite[Proposition~3.1.13]{FaureFrolicher} that $G$ is a finitary simple matroid. With a slight abuse of notation, we let $G$ denote both the projective geometry and its associated matroid. Then we let $G_0$ denote the simple pointed matroid associated to the simple matroid $G$, i.e.~its image under the functor $\sMat \to \sMat_\bullet$. 

We claim that this assignment extends to a functor
\begin{align}\label{eq:projective functor}
\begin{split}
\Proj &\to \sMat_\bullet, \\
G &\mapsto G_0.
\end{split}
\end{align}
Indeed, if $f \colon G \dashrightarrow H$ is a morphism of projective geometries, it extends to a function on pointed sets $f_0 \colon G_0 \to H_0$ by $f_0(\ker f \sqcup \{0\}) = 0$. Then if $E \subseteq H$ is a subspace of $H$, we have 
\[
f_0^{-1}(E \sqcup\{0\}) = f^\sharp(E) \sqcup \{0\},
\]
where $f^\sharp(E) \subseteq G$ is a subspace as discussed above. Thus the preimage of any closed set in $H_0$ is closed in $G_0$, so that $f_0$ is a morphism of matroids. The remaining axioms of functoriality for $f \mapsto f_0$ are readily verified.

\begin{proposition}
The functor~\eqref{eq:projective functor} sending a projective geometry to the corresponding pointed matroid yields an equivalence of categories 
\[
\Proj \overset{\sim}{\longrightarrow} \pMat_\bullet. 
\]
Under this correspondence, if $f \in \Proj(G,H)$ has $\ker f = N$, then the kernel of $f_0 \in \pMat_\bullet(G_0, H_0)$ is $N \sqcup \{0\}$.
\end{proposition}

\begin{proof}
The essential image of the functor consists of exactly the projective pointed matroids thanks to~\cite[Corollary~3.3.8]{FaureFrolicher}, which shows that that for any set $G$, there is a bijection between the structures of a projective geometry on $G$ and the structures of a projective matroid on $G$. The functor is faithful since the diagram
\[
\begin{tikzcd}
\Proj \ar[d, hookrightarrow] \ar[r] & \pMat_\bullet \ar[d, hookrightarrow] \\
\Par \ar[r, "\sim"] & \Set_\bullet 
\end{tikzcd}
\]
commutes, where the vertical arrows are faithful forgetful functors. Lastly, the functor is full by the characterization~\cite[Proposition~6.2.3]{FaureFrolicher} of morphisms of projective geometries as exactly those partial functions $f \colon G \dashrightarrow H$ such that $f^\sharp(E)$ is a subspace of $G$ whenever $E$ is a subspace of $H$.
\end{proof}

If $M$ is a commutative mosaic and $S \subseteq M$, we let $\langle S \rangle$ denote the strict submosaic of $M$ generated by $S$. 

\begin{proposition}\label{prop:projective closure}
Let $G$ be a pointed projective matroid. For any $S \subseteq G$, we have $C(S) = \langle S \rangle$.  Thus the closed subsets of $G$ are precisely the strict submosaics of $G$.
\end{proposition}

\begin{proof}
Note that $S \cup \{0\} \subseteq C(S)$ and that if $\alpha, \beta \in C(S)$ then
\[
\alpha + \beta \subseteq C(\alpha, \beta) \subseteq C(C(S)) = C(S).
\]
Thus $C(S)$ is a strict submosaic containing $S$, and it follows that $\langle S \rangle \subseteq C(S)$. 

Conversely we wish to show that $C(S) \subseteq \langle S \rangle$. Because $G$ is finitary, we have $C(S) = \bigcup C(T)$ where $T$ ranges over the finite subsets of $S$. Thus it suffices to consider the case where $S$ is finite. The claim is easily verified if $S$ is empty or a singleton, so we may assume $S = \{x_1, \dots, x_n\}$ and assume for inductive hypothesis that the claim holds for all sets of cardinality at most~$n-1$. 

Let $\alpha \in C(S) = C(\{x_1\} \cup \{x_2, \dots, x_n\})$. It follows from the projective axiom that $\alpha \in C(x_1, \beta)$ for some $\beta \in C(x_2, \dots, x_n)$. By inductive hypothesis, $\beta \in \langle x_2, \dots, x_n \rangle \subseteq \langle S \rangle$. If $\alpha \in \{x_1, \beta, 0\}$ then certainly $\alpha \in \langle S \rangle$. Otherwise $x_1$ and $\beta$ are nonzero and distinct (else $\alpha \in C(x_1, \beta) \subseteq \{x_1, \beta, 0\}$), in which case
\[
\alpha \in C(x_1, \beta) \setminus \{x_1,\beta,0\} = x_1 + \beta \subseteq \langle S \rangle
\]
because $x_1, \beta \in \langle S \rangle$. Thus we find $C(S) \subseteq \langle S \rangle$ as desired.
\end{proof}

\begin{theorem}\label{thm:projective embedding}
The functor $\sMat_\bullet \to \cMsc$ restricts to a fully faithful functor on the full subcategory of pointed projective matroids, yielding a fully faithful functor
\[
\Proj \cong \cat{pMat}_\bullet \hookrightarrow \cMsc.
\]
\end{theorem}

\begin{proof}
By Theorem~\ref{thm:matroid mosaics}, the functor $\pMat_\bullet \to \cMsc$ is faithful. To see that it is full, let $G$ and $H$ be pointed projective matroids and let $f \in \cMsc(G,H)$. We wish to show that $f$ is a strong map of matroids. Let $E \subseteq H$ be a closed subset. By Proposition~\ref{prop:projective closure} this means that $E$ is a strict submosaic of $H$. Because $f$ is a morphism of mosaics, it is straightforward to verify that $f^{-1}(E)$ is a strict submosaic of $G$. (Alternatively, by Proposition~\ref{prop:strict classifier} we have $E = \ker \alpha$ for some $\alpha \in \cMsc(H,\K)$, and then one may deduce that $f^{-1}(E) = \ker (\alpha f)$ is strict.) But then $f^{-1}(E)$ is closed in $G$ by Proposition~\ref{prop:projective closure} again, proving that $f$ is a strong map as desired.
\end{proof}

In closing, we mention one potential avenue for further work. While the embedding of simple pointed matroids in to commutative mosaics from Theorem~\ref{thm:matroid mosaics} is faithful in general and full when restricted to $\pMat_\bullet$, it would be interesting to better understand its failure to be full in general. There are several questions whose answer could shed light on this problem. How can we characterize the mosaics in the essential image of this functor? Which morphisms of mosaics are in the image of the embedding $\sMat_\bullet \to \cMsc$? Can the mosaics in the essential image be equipped with a natural extra structure to factor this embedding through a fully faithful functor?

\bibliographystyle{amsplain}
\bibliography{hyperstructures-v3}

\end{document}